\documentclass{amsart}

\usepackage{amsmath,amssymb,amsthm,amsfonts,graphicx,color,enumerate,float,epsfig}

\usepackage[all]{xy}
\newcommand{\cd}[1]{\begin{equation*}{\xymatrix{#1}}\end{equation*}}
\newcommand{\cdlabel}[2]{\begin{equation}\label{#1}{\xymatrix{#2}}\end{equation}}

\newcommand{\invlim}{\underleftarrow{\lim}}
\newcommand{\dirlim}{\underrightarrow{\lim}}
\newcommand{\twid}[1]{\tilde{#1}}
\newcommand{\rrangle}{\rangle\negthinspace\rangle}
\newcommand{\llangle}{\langle\negthinspace\langle}
\newcommand{\linj}{\mathrm{loginj}}
\newcommand{\lsurj}{\mathrm{logsurj}}
\newcommand{\consum}{\mathrel{\#}}
\newcommand{\mc}{\mathcal}

\newcommand{\join}{\ast}

\newcommand{\coker}{\mathrm{coker}}

\def\co{\colon\thinspace}
\def\R{\mathbb{R}}
\def\Z{\mathbb{Z}}

\def\H{\mathbb{H}}
\def\N{\mathbb{N}}

\def\E{\mathbb{E}}

\newtheorem{theorem}{Theorem}[section]
\newtheorem{lemma}[theorem]{Lemma}
\newtheorem{corollary}[theorem]{Corollary}
\newtheorem{proposition}[theorem]{Proposition}
\newtheorem{question}[theorem]{Question}

\newtheorem{claim}[theorem]{Claim}

\theoremstyle{definition}
\newtheorem{remark}[theorem]{Remark}
\newtheorem{definition}[theorem]{Definition}

\newtheorem{observation}[theorem]{Observation}
\newtheorem{useful}[theorem]{Useful facts}
\newtheorem{assumption}[theorem]{Standing assumption}

\def\version{\today}

\begin{document}

\author[K. Fujiwara]{Koji Fujiwara}
\address{Graduate School of Information Science, Tohoku University}
\email{fujiwara@math.is.tohoku.ac.jp}

\author[J. F. Manning]{Jason Fox Manning}
\address{Department of Mathematics, University at Buffalo, SUNY}
\email{j399m@buffalo.edu}

\title{CAT$(0)$ and CAT$(-1)$ fillings of hyperbolic manifolds}

\thanks{\version}

\begin{abstract}
We give new examples of hyperbolic and relatively hyperbolic
groups of cohomological dimension
$d$ for all $d\geq 4$ (see Theorem \ref{t:examples}).  These examples result from applying
CAT$(0)$/CAT$(-1)$ filling constructions (based on singular doubly
warped products) to finite volume hyperbolic
manifolds with toral cusps.

The groups obtained have a number of interesting properties, which are
established by analyzing their boundaries at infinity by a kind of
Morse-theoretic technique, related to but distinct from ordinary and
combinatorial Morse theory (see Section \ref{s:visbound}).
\end{abstract} 

\maketitle

\tableofcontents

\section{Introduction}
In this paper we study some generalizations of the 
Gromov-Thurston $2\pi$ theorem \cite{BlHo}.  Informally, the
$2\pi$ theorem states that ``most'' Dehn fillings of a hyperbolic
$3$--manifold with cusps admit negatively curved metrics.  Moreover,
these negatively curved metrics are close approximations of the metric
on the original cusped manifold.
Group-theoretically, the fundamental groups of these fillings are
relatively hyperbolic (defined below).  If all cusps of the original manifold were
filled, the fundamental group of the filling is hyperbolic;  if
moreover the original manifold was finite volume, the filling is an
aspherical closed $3$--manifold.

In the world of coarse geometry, the correct generalizations of
fundamental groups of hyperbolic manifolds with and without cusps are
\emph{relatively hyperbolic} and \emph{hyperbolic} groups,
respectively.  These generalizations were introduced by Gromov in
\cite{gromov:hg}.  We give a quick review of the definitions:
A metric space is
\emph{$\delta$--hyperbolic} for $\delta>0$
if all its geodesic triangles are
$\delta$--thin, meaning each side of the triangle is contained in the
$\delta$--neighborhood of the other two.  A space is \emph{Gromov
  hyperbolic} if it is $\delta$--hyperbolic for some $\delta$.
Groups which are hyperbolic
or relatively hyperbolic are those which have particular kinds of
actions on Gromov hyperbolic spaces.  If $G$ acts by isometries on a
Gromov hyperbolic space $X$ properly and cocompactly, then $G$ is said to be
\emph{hyperbolic}.  
A \emph{geometrically
finite} action of a group $G$ on a Gromov hyperbolic space $X$ is one
which satisfies the following conditions:
\begin{enumerate}
\item $G$ acts properly and by isometries on $X$.
\item There is a family $\mc{H}$ of disjoint \emph{horoballs}
  (sub-level sets of so-called \emph{horofunctions}), preserved by the
  action.  
\item The action on $X\setminus \bigcup \mc{H}$ is cocompact.
\item The quotient $X/G$ is quasi-isometric to a wedge of rays.
\end{enumerate}
(See \cite{gromov:hg} for more details.)
If $G$ acts on $X$ geometrically finitely,
then $G$ is said to be \emph{relatively hyperbolic}, relative to  
the \emph{peripheral} subgroups, which are the
stabilizers of individual horoballs.  (Typically the list of peripheral
subgroups is given by choosing one such stabilizer from each conjugacy
class.)

In \cite{GM} and \cite{Os}, group-theoretic analogues of
the $2\pi$ theorem were proved.  These theorems roughly state that
if one begins with a relatively hyperbolic group, and kills off normal
subgroups of the peripheral subgroups which contain no ``short''
elements, then the resulting group is also relatively hyperbolic, and
``close'' to the original group in various ways.  From this one
recovers a weaker group-theoretic version of the $2\pi$ theorem, as the
fundamental group of a hyperbolic $3$--manifold is relatively
hyperbolic, relative to the cusp groups;  the operation of Dehn
filling acts on fundamental groups by killing the cyclic subgroups
generated by the filling slopes.

In the context of the original $2\pi$ theorem, we begin with a group
which has such an action on $\H^3$, which is not only
$\delta$--hyperbolic, but which is actually CAT$(-1)$.
(See Section \ref{s:prelim} for the
definition of CAT$(\kappa)$ for $\kappa\in \R$.  For now, just note
that CAT$(-1)$ implies Gromov hyperbolic.)  Moreover, the
group acts cocompactly on a ``neutered'' $\H^3$, which is CAT$(0)$
with isolated flats.  The fundamental group of the filled manifold
acts properly and cocompactly on a CAT$(-1)$ $3$--manifold.

For $\kappa\leq 0$, we say that a group is \emph{CAT$(\kappa)$} if it
acts isometrically, properly, and cocompactly on some CAT$(\kappa)$ space.
We will say that a group is \emph{relatively CAT$(-1)$} if it has a
geometrically finite action
on a CAT$(-1)$ space.  A group is
\emph{CAT$(0)$ with isolated flats} if it acts isometrically, properly, and
cocompactly on some CAT$(0)$ space with isolated flats.  (It follows
from a result of Hruska and Kleiner \cite{HK} that horoballs can
always be added to such a space to make it $\delta$--hyperbolic for
some $\delta$.  In particular, a group which is CAT$(0)$ with isolated
flats is relatively hyperbolic.)
It is natural
to wonder to what extent the main results of \cite{GM} and \cite{Os}
can be strengthened in the presence of these stronger conditions.

\begin{question}
 If one performs relatively hyperbolic Dehn filling on a
  relatively CAT$(-1)$ group, is the filled group always relatively
  CAT$(-1)$?
\end{question}
\begin{question}
 If one performs relatively hyperbolic Dehn filling on a CAT$(0)$
  with isolated flats group, is the filled group always CAT$(0)$ with
  isolated flats?
\end{question}

These questions form a part of the motivation for this paper; we
answer them positively in some special cases. 
The utility of these answers is illustrated by
using the additional CAT$(0)$ structure of the quotient groups to
obtain information about the filled groups which is less accessible
from the coarse geometric viewpoint.

\subsection{Outline}
In Section \ref{filling}, we describe the main results.  In Section
\ref{s:prelim}, we recall the definition of warped product metric and
an important result of Alexander and Bishop.  We also discuss the
space of directions and the logarithm map from a CAT$(0)$ space to the
tangent cone at a point.  In Section \ref{s:basic}, we extend original
construction of the $2\pi$ theorem to our setting, using doubly warped
product constructions to give
locally CAT$(0)$ and locally CAT$(-1)$ models for our fillings.
Section \ref{s:visbound} contains the ``Morse theoretic'' arguments
which give information about the boundary at infinity of the
fundamental groups of our fillings.  In Section
\ref{s:questions}, we pose some questions.

\subsection{Acknowledgments}
Fujiwara started working on this project
during his visit at Max Planck Institute in Bonn
in the fall of 2005. He has benefited from 
conversation with Potyagailo.
The joint project with Manning started when he visited 
Caltech in April 2006.
He thanks both institutions for their hospitality. He is partially
supported by Grant-in-Aid for Scientific Research (No. 19340013). 

Manning thanks Ric Ancel for explaining how to prove Claim
\ref{horiz} in the proof of Proposition \ref{p:pinch}, and thanks Noel
Brady and Daniel Groves for useful conversations.  He was
partially supported by NSF grants DMS-0301954 and DMS-0804369.

\section{Statements of results}\label{filling}

\subsection{Cones and fillings}
We first establish some notation:
\begin{definition}[Partial cone]\label{def.cone}
Let $N$ be a flat torus of dimension $n$, and let $T$ be a totally
geodesic $k$--dimensional submanifold in $N$, where $1\leq k\leq n$.  The torus $N$
has Euclidean universal cover $\E^n\to N$; let $\twid{T}\cong \E^k$ be a
component of the preimage of $T$ in $N$.  The parallel copies of
$\twid{T}$ are the leaves of a fibration
$\E^n\stackrel{\twid{\pi}}{\longrightarrow}\E^{n-k}$.  This fibration covers
a fibration
\[ N \stackrel{\pi}{\longrightarrow} B\]
of $N$ over some $(n-k)$--torus $B$.  We define the \emph{partial
  cone} $C(N,T)$ to be the mapping cylinder of $\pi$, i.e.
\[ C(N,T) = N \times [0,1] / \sim, \]
where $(t_1,1) \sim (t_2,1)$ if $\pi(t_1)=\pi(t_2)$.  

We refer to $N\times 0$ as the \emph{boundary} of $C(N,T)$, and to
$V(N,T):=N\times 1/\sim$ as the \emph{core}.
\end{definition}

\begin{figure}[htbp]
\begin{center}
\input{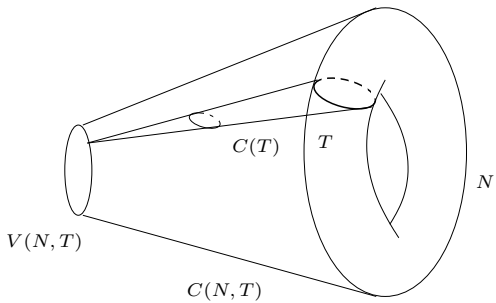}
\caption{cone $C(N,T)$}
\end{center}
\end{figure}

\begin{remark}
  The space $C(N,N)$ is the cone on $N$.  In the case that $N$ is a
  metric product $N = T\times B$, then $C(N,T)$ can be canonically
  identified with $C(T)\times B$.  (Although the identification is not
  canonical in general, $C(N,T)$ is always homeomorphic to $C(T)\times
  B$ for some torus $B$ of complementary dimension to that of $T$.  In
  particular, the core $V(N,T)$ is always homeomorphic to a torus of
  dimension $\dim(N)-\dim(T)$.)  The partial cone $C(N,T)$ is a
  manifold with boundary only when $\dim(T) = 1$.
\end{remark}

\begin{definition}
  Let $M$ be a hyperbolic $(n+1)$--manifold.  A \emph{toral cusp} of
  $M$ is a neighborhood $E$ of an end of $M$ so that 
  \begin{enumerate}
    \item $E$ is homeomorphic to the product of a torus with the
      interval $(-\infty,0]$, and
    \item the induced Riemannian metric on $\partial E$ is flat.
  \end{enumerate}
\end{definition}

For the rest of the section we make the following:
\begin{assumption}\label{standing}
  $M$ is a hyperbolic $(n+1)$--manifold of finite volume containing
  disjoint toral cusps $E_1,\ldots,E_m$, and no other ends.
\end{assumption}

Let $\overline{M}\subset
M$ be the manifold with boundary obtained by removing the interiors of
the cusps.  Let $N_1,\ldots,N_m$ be the components of $\partial \overline{M}$.

\begin{definition}[Filling, $2 \pi$--filling]\label{d:filling}
For each component $N_i$ of $\partial \overline{M}$ choose an embedded totally
geodesic torus $T_i$ of dimension $k_i$.  Form $C(N_i,T_i)$ as in
the previous definition.  We use the canonical identification $\phi_i$
of $\partial C(N_i,T_i)$ with $N_i$ to form the space
\[ M(T_1,\ldots,T_m):=
\overline{M}\bigcup_{\phi_1\sqcup\cdots\sqcup\phi_m}(C(N_1,T_1)\sqcup\cdots\sqcup C(N_m,T_m)).\]
We say that $M(T_1,\ldots,T_m)$ is \emph{obtained from $M$ by filling
  along $\{T_1,\ldots,T_m\}$}.  We refer to the cores of the
$C(N_i,T_i)$ as the \emph{filling cores}.

If no $T_i$ contains a geodesic of length $2\pi$ or less, we say that
$M(T_1,\ldots,T_m)$ is a \emph{$2\pi$--filling of $M$}.
\end{definition}

\begin{remark}
The space $M(T_1,\ldots,T_m)$ described above
is homeomorphic to a manifold if and only if every filling core has
dimension exactly $n-1$.  In any case $M(T_1,\ldots,T_m)$ is a
pseudomanifold of dimension $(n+1)$, and the complement of the filling
cores is homeomorphic to $M$.
\end{remark}

On the level of fundamental groups, $\pi_1(M(T_1,\ldots,T_m))$ is
obtained from $\pi_1(M)$ by adjoining relations representing the
generators of some direct summands of the cusp subgroups;  namely, for
each $i\in \{1,\ldots,m\}$, one adjoins relations corresponding to the
generators of $\pi_1(T_i)\lhd\pi_1(E_i) = \pi_1(N_i)$. 

\subsection{Main results}
Our main theorem says that under hypotheses exactly analogous to those
used in the $2\pi$ theorem, a space obtained by filling as above can
be given a nice locally CAT$(0)$ metric; moreover, if all the filling
cores are zero or one dimensional, the metric can be chosen to be
locally CAT$(-1)$.  We remind the reader that, under the condition
\ref{standing},  $M$ is assumed to be a
hyperbolic $(n+1)$--manifold with only toral cusps.

\begin{theorem}\label{t:filling}
Let $n\geq 2$ and assume the condition \ref{standing}.

If $M(T_1,\ldots,T_m)$ is a $2\pi$--filling of $M$,
then there is a complete path metric $d$ on
$M(T_1,\ldots,T_m)$ satisfying the 
following:
\begin{enumerate}
\item\label{fill1} The path metric $d$ is the completion of a path metric induced
  by a negatively curved Riemannian metric on the complement of the filling cores.
\item\label{fill2} The path metric $d$ is locally CAT$(0)$.
\item\label{fill3} If for every $i$, $T_i$ has codimension at most $1$ in $\partial
  E_i$,
  then $d$ is locally CAT$(-\kappa)$ for some $\kappa>0$.
\end{enumerate}
\end{theorem}
The metric constructed in Theorem \ref{t:filling} is identical to the
original hyperbolic metric, away from a neighborhood of the filling cores.  We
will prove the following in the course of showing Theorem
\ref{t:filling}.
\begin{proposition}\label{p:hyper}
  Let $n\geq 2$ and assume the condition \ref{standing}, and let 
  $M(T_1,\ldots,T_m)$ be a
  $2\pi$--filling of $M$.  Let 
  \[ \overline{M} = M\setminus \bigcup_{i=1}^m E_i \]
  be as described in
  Definition \ref{d:filling}.
  Choose $\lambda>0$ so that the
  $\lambda$--neighborhood
  of $\bigcup_{i=1}^mE_i$ is embedded in $M$.  
  Let $\overline{M}'$ be the complement in $M$ of the 
  $\frac{\lambda}{2}$--neighborhood of
  $\bigcup_{i=1}^mE_i$.  We thus have inclusions
  $\overline{M}'\subset \overline{M}\subset M$, and
  $\overline{M}'\subset \overline{M}\subset M(T_1,\ldots,T_m)$. 

  The Riemannian metric described in Theorem \ref{t:filling} is equal
  to the hyperbolic metric, when restricted to $\overline{M}'$.
\end{proposition}

\begin{remark}
  Let $G = \pi_1(M(T_1,\ldots,T_m))$ be the fundamental group of some
  $2\pi$--filling of $M$.
  It follows from Theorem \ref{t:filling} that
  $H_{n+1}(G) \cong H_{n+1}(M(T_1,\ldots,T_m))$ is cyclic.  One can ask what the Gromov
  norm of a generator is, and define this number to be the
  \emph{simplicial volume} of $M(T_1,\ldots,T_m)$ (or of $G$).  In case
  $M(T_1,\ldots,T_m)$ is a manifold, this definition is the same as
  Gromov's in \cite{gromov:vbc}.  In the setting of the original
  $2\pi$ theorem, both $M$ and $M(T_1,\ldots,T_m)$ are hyperbolic
  manifolds, and it was shown by Thurston that the simplicial volume of
  $M(T_1,\ldots,T_m)$ is bounded above by the hyperbolic volume of
  $M$, divided by a constant \cite{Th}.  We extend Thurston's result to higher
  dimensions in \cite{FM2}.
\end{remark}

\begin{remark}\label{r:specialcases}
Special cases of Theorem \ref{t:filling} appear in 
\begin{itemize}
\item the Gromov-Thurston
  $2\pi$--theorem \cite{BlHo}, in which $n=2$ and each
  filling core is a circle,
\item Schroeder's generalization of the $2\pi$ theorem to $n>2$, with
  filling cores equal to 
  $(n-1)$--dimensional tori \cite{Sc} (This is the case in which every
  $T_i$ is one-dimensional, and the only case in which
  the filling $M(T_1,\ldots,T_m)$ is a manifold.), and
\item Mosher and Sageev's construction in \cite{MS} of $(n+1)$--dimensional
  hyperbolic groups which are not aspherical manifold groups; here the
  filling cores are points.
\end{itemize}
The construction of Mosher and Sageev is based on a discussion in
\cite[7.A.VI]{G93}, and the examples constructed by Mosher and Sageev
answered a question which appeared in an early version of Bestvina's
problem list \cite{Bes}.  The new examples of the current paper fall in some
sense \emph{between} those of Schroeder, and those of Mosher-Sageev.
\end{remark}

Our construction gives the following corollary, which can be deduced
from Theorem \ref{t:filling} and Proposition \ref{p:iflats}.
\begin{corollary}\label{c:cwif}
Let $n\geq 2$ and assume the condition \ref{standing}.

Let $M(T_1,\ldots,T_m)$ be a $2\pi$--filling of $M$, and let 
$G$ be the fundamental group of $M(T_1,\ldots,T_m)$.
Then $G$ is CAT$(0)$ with isolated flats.  If each $T_i$
has dimension $n$ or 
$n-1$, then $G$ is CAT$(-1)$.
\end{corollary}

Thus Theorem \ref{t:filling} says that given a cusped hyperbolic
$(n+1)$--manifold with $n\geq 2$,
there are many ways to fill its cusps to get a space whose universal
cover is CAT$(-\kappa)$ for $\kappa>0$.  If $n\geq 3$, there are many
ways to fill its cusps to get a space whose universal cover is 
CAT$(0)$ with isolated flats.  
The following proposition says
that different fillings generally have different fundamental groups:
\begin{proposition}\label{p:different}
Let $n\geq 2$ and assume the condition \ref{standing}.

For each $i\in \N$ choose a
filling $M_i = M(T_1^i,\ldots,T_m^i)$ of $M$, and suppose that the fillings
$M_i$ satisfy:
For each 
$1\leq j\leq m$ and each element $\gamma$ of
$\pi_1(E_j)\setminus\{1\}$, the set  
$\{i\mid \gamma\in \pi_1(T_j^i)\}$
is finite.  Then for any $k$, there are only finitely many $i\in \N$
so that $\pi_1(M_i)\cong \pi_1(M_k)$. 
\end{proposition}
\begin{proof}
We argue by contradiction.
The hypotheses imply that all but finitely many fillings $M_i$ are
$2\pi$--fillings.  By passing to a subsequence we may suppose that all
the $\pi_1(M_i)$ are isomorphic to some fixed $G$.
Applying Corollary \ref{c:cwif} and 
\cite[Theorem 1.2.1]{HK},
$G$ is relatively hyperbolic with respect to a collection of
free abelian subgroups.  If $\Gamma = \pi_1(M)$ we have a sequence of
surjections 
\[ \Gamma \stackrel{\phi_i}{\longrightarrow}G \]
obtained by killing subgroups $K_j^i$
of the peripheral subgroups of $\Gamma$
corresponding to the tori $T_j^i$.  Since the injectivity radii of
these tori are going to infinity, the length (with respect to a fixed
generating set for $\Gamma$)
of a shortest nontrivial element of $\cup_jK_j^i$ is also going to
infinity.
It follows from Corollary 9.7 of \cite{GM} that
the \emph{stable kernel} 
\[ \underrightarrow{\ker}\{ \phi_i \} = \{ \gamma \in \Gamma \mid
\phi_i(\gamma)=1\mbox{ for almost all }\phi_i\}\]
is trivial.  Work of Groves then shows that $\Gamma$
is either abelian or splits over an abelian group \cite[Proposition
  2.11]{Gr}.  
But $\Gamma$ is
the fundamental group of a finite volume hyperbolic manifold of
dimension at least $3$, so it cannot split over a virtually abelian
subgroup (see, e.g. \cite[Theorem 1.3.(i)]{Be07}).
\end{proof}

  We will show 
(Corollary
\ref{c:notpdn}) that no $2\pi$--filling of an $(n+1)$--dimensional hyperbolic
manifold 
with filling cores of dimension $<n$ 
has the same fundamental group as a closed hyperbolic
manifold (or even a closed aspherical manifold) for $n\geq 3$.  This is in
sharp contrast to the situation for hyperbolic $3$--manifolds.
Combining Corollary \ref{c:notpdn} and Proposition \ref{p:different}
immediately yields the following:
\begin{theorem}\label{t:examples}
Let $n\geq 3$ and assume the condition \ref{standing}.
There are infinitely many $2\pi$--fillings of $M$ whose
fundamental groups are non-pairwise-isomorphic, torsion-free,
 $(n+1)$--dimensional hyperbolic groups, each of which acts
geometrically on some $(n+1)$--dimensional CAT$(-1)$ space, and none of
which is a closed aspherical manifold group.

If $n\geq 4$, there are infinitely many $2\pi$--fillings of
$M$ whose fundamental groups are non-pairwise-isomorphic,
torsion-free,  $(n+1)$--dimensional relatively hyperbolic groups, each of which
acts geometrically on some $(n+1)$--dimensional CAT$(0)$ space with
isolated flats, and none of which is a closed aspherical manifold
group.  Flats of any dimension up to $(n-2)$ occur in infinitely many
such fillings.
\end{theorem}

In \cite{JS2}, Januszkiewicz and {\'S}wi{\c{a}}tkowski introduce
\emph{systolic groups}, in part as a means of giving examples of 
non-manifold $n$--dimensional
hyperbolic groups for all $n$ (see \cite{CD,JS1} for other
approaches).  Recently, Osajda has shown that 
systolic groups are never simply connected at infinity
\cite{Osa}.  It follows from our analysis of the group boundary in
Section \ref{s:visbound} that the high-dimensional groups obtained in
Theorem \ref{t:examples} are simply connected at infinity whenever
no filling cores are points.  In particular, these groups are not
systolic (Corollary \ref{c:notsys}).

\section{Preliminaries}\label{s:prelim}
For general results and terminology for CAT$(\kappa)$ spaces we refer
to Bridson and Haefliger \cite{BrHa}.  For convenience, we remind the
reader of the definition.  Let $\kappa\in \R$, and let $S_\kappa$ be
the complete simply connected surface of constant curvature $\kappa$.
Let $R_\kappa\le \infty$ be the diameter of $S_\kappa$.  Now let $X$ be a
length space.  (See \cite{BrHa,BBI} for
more on length spaces.)
If $\Delta\subset X$ is
a geodesic triangle in $X$ whose perimeter is bounded above by
$2R_\kappa$, with corners $x$, $y$, and $z$, then there is a
\emph{comparison triangle} $\bar{\Delta}\subset S_\kappa$; i.e., the
corners of $\bar{\Delta}$ can be labelled $\bar{x}$, $\bar{y}$, and
$\bar{z}$, and there is a map from $\Delta$ to $\bar{\Delta}$ which is
an isometry restricted to any edge of $\Delta$, taking $x$ to $\bar{x}$, and so
on (See Figure \ref{f:compare}).  
\begin{figure}[htbp]\label{f:compare}
  \begin{center}
    \input{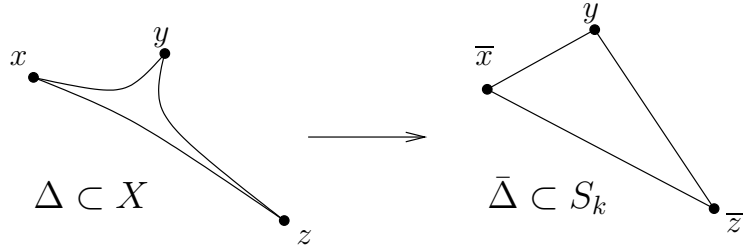}
  \caption{Comparison for CAT$(k)$.}
  \end{center}
\end{figure}
If this map never decreases distance on any triangle $\Delta$
whose perimeter is bounded above by $2R_\kappa$, we say that $X$ is
\emph{CAT$(\kappa)$}.  Put informally, $X$ is CAT$(\kappa)$ if its
triangles are ``no fatter'' than the comparison triangles in $S_\kappa$.
If every point in $X$ has a neighborhood which is CAT$(\kappa)$, then
we say that $X$ is \emph{locally CAT$(\kappa)$}.  If $X$ is locally
CAT$(0)$, then we say that $X$ is \emph{nonpositively curved}.  For
$\kappa\leq 0$, the
universal cover of a complete, locally CAT$(\kappa)$ space is
CAT$(\kappa)$, by the Cartan-Hadamard Theorem.

\subsection{Warped products}\label{ss:wp}
We take this definition from \cite[3.6.4]{BBI}.
\begin{definition}
  Given two complete length spaces $(X,d_X)$ and $(Y,d_Y)$, and a continuous
  nonnegative  
  $f\co X\to [0,\infty)$, we can define the \emph{warped product of $X$
  with $Y$, with warping function $f$} to be the length space based on the
  following method of measuring lengths of paths.
  Given any closed interval $[a,b]$ and any lipschitz path $\gamma\co I\to
  X\times Y$, with $\gamma(t) = (\gamma_1(t),\gamma_2(t))$, the
  quantities $|\gamma_1'(t)|$ and $|\gamma_2'(t)|$ are well defined
  almost everywhere on $I$ (see \cite[2.7.6]{BBI}).  Define the length
  of $\gamma$, $l(\gamma)$ to be
\[ l(\gamma) = \int_a^b\sqrt{|\gamma_1'(t)|^2+f^2(\gamma_1(t))|\gamma_2'(t)|^2}dt .\]
\end{definition}

The quotient map
$X\times Y\to X\times_f Y$ is injective exactly when $f$ has no
zeroes.  If $f(x)=0$ then $\{x\}\times Y$ is identified to a point in
$X\times_fY$.

Often we will have $X \subseteq \R$, or $X = I\times_g Z$ for
$I\subseteq \R$, and $f$ depending only on $r\in I$.  In this case, we
will usually write $X\times_f Y$ as $X\times_{f(r)}Y$.  
For example:
\begin{itemize}
\item The euclidean plane is isometric to $[0,\infty)\times_r S^1$.
  The subset $\{R\}\times S^1$ is a circle of radius $R$.
\item The hyperbolic space $\H^{n+1}$ is isometric to
  $\R\times_{e^r}\E^n$.  The subset $\{R\}\times \E^n$ is a
  horosphere. 
\end{itemize}

We will often abuse notation by using
points in $X\times Y$ to
refer to points in $X\times_f Y$, even when the quotient map is not
injective.  For example the cone point of the Euclidean cone 
$[0,\infty)\times_r  X$ might be referred to as $(0,x)$ for some $x\in
  X$, or just as $(0,-)$.

In case the metrics $d_X$ and $d_Y$ are induced by Riemannian metrics
$g_X$ and $g_Y$ and $f\co X\to \R_{>0}$ is smooth, the warped product
metric on $X\times_f Y$ is equal 
to the path metric induced by a Riemannian metric $g$, given 
at a point $(x,y)$ by the formula 
\[ g = g_X + f^2(x) g_Y. \]

Curvature bounds on warped product spaces are studied by Alexander and
Bishop in
\cite{AB}, where quite general sufficient conditions
are given for a warped product to have curvature bounded from above or
below.
In order to state the results from \cite{AB} which we need, we first
recall some definitions.
\begin{definition}
Let $K \in \R $.
Let $\mathcal{F}K$ be the set of solutions to the differential
equation $g'' + Kg = 0$.  A  function $f$ 
satisfying $f \leq g$ for any function $g$ in $\mathcal{F}K$
agreeing with $f$ at two points which are sufficiently close together
is said to \emph{satisfy the differential inequality $f''+Kf\geq 0$ in
  the barrier sense.}
\end{definition}
\begin{definition}
Let $K\in \R$, and let $B$ be a geodesic space.  A continuous function
$f\co B\to \R$ is \emph{$\mathcal{F}K$--convex} if its composition with
any unit-speed geodesic satisfies the differential inequality
$f''+Kf\geq 0$ in the barrier sense.
\end{definition}
We collect here some useful facts about $\mathcal{F}K$--convexity:
\begin{useful}\label{convex}{\ }
\begin{enumerate}
\item A function is $\mathcal{F}0$--convex if and only if it is
  \emph{convex} in the usual sense.
\item\label{c-fort} A non-negative $\mathcal{F}K$--convex function is also
  $\mathcal{F}K'$--convex for any $K'\geq K$.
\item \cite[Theorem 1.1(1A)]{AB2} If $B$ is CAT$(-1)$ and $x\in B$,
  then $\cosh(d(\cdot,x))$ is a $\mathcal{F}(-1)$--convex function on $B$.
\item\label{c-dist} \cite[Theorem 1.1(3A)]{AB2} If $A$ is a convex subset of a
  CAT$(0)$ space $B$, then $d(\cdot,A)$ is a convex function on $B$.
\item\label{c-sinh} \cite[Theorem 1.1(4A)]{AB2} If $A$ is a convex subset of a CAT$(-1)$ space $B$,
  then $\sinh(d(\cdot,A))$ is a $\mathcal{F}(-1)$--convex function on
  $B$.
\end{enumerate}
\end{useful}

We now state a special case of a more general result in
\cite{AB}.  

\begin{theorem}\label{t:AB}\cite[Theorem 1.1]{AB}
Let $K\leq 0$, and let $K_F\in \R$.
Let $B$ and $F$ be complete locally compact CAT$(K)$ and CAT$(K_F)$ spaces,
respectively. 
Let $f\co B\to [0,\infty)$ be $\mathcal{F}K$--convex.  Let $X =  f^{-1}(0)$.

Suppose further that either 
\begin{enumerate}
\item $X=\emptyset$ and $K_F\leq K(\inf f)^2$, or
\item $X$ is nonempty and $(f\circ\sigma)'(0^+)^2 \geq K_F$ whenever
  $\sigma\co [0,x]\to B$ is a shortest (among all paths starting in
  $X$) unit-speed geodesic from $X$ to a point $b\in B\smallsetminus X$. 
\end{enumerate}
Then the warped product $B\times_f F$ is CAT$(K)$.
\end{theorem}

\subsection{The space of directions and the tangent cone}
Suppose that $X$ is a length space, and $p\in X$.  Recall that the
\emph{space of directions at $p$} is the collection of geodesic
segments issuing from $p$, modulo the equivalence relation:
$[p,x]\sim[p,y]$ if the Alexandroff angle between $[p,x]$ and $[p,y]$
at $p$ is zero.  The Alexandroff angle at $p$ makes the space of
directions, $\Sigma_p(X)$, into a metric space.  If $X$ is a Riemannian
manifold, then $\Sigma_p(X)$ is always
a round sphere of diameter $\pi$.
The \emph{tangent cone at $p$} is the Euclidean cone on the space of
directions at $p$:
\[ C_p(X)\cong [0,\infty)\times_r\Sigma_p(X) \]
We will sometimes just write $\Sigma_p$ and $C_p$ if the space $X$ is
understood.  If $X$ is a Riemannian manifold near the point $p$, then
the tangent cone $C_p$ can be canonically identified with the tangent
space of $X$ at $p$.

In our setting $X$ will always be geodesically complete and locally
compact.  With these assumptions, if $U$ is any open neighborhood of
$p\in X$, then $C_p(X)$ is isometric to the pointed Gromov-Hausdorff
limit of spaces $U_n$ isometric to $U$ with the metric multiplied by
factors $R_n$ going to infinity (see for example \cite[Theorem
  9.1.48]{BBI}).

Alexander and Bishop \cite[page 1147]{AB} observe that the space of
directions at a point in a warped product is determined by the spaces
of directions in the factors together with infinitesimal information
about the warping functions.  If $f\co B\to \R$ is a convex function
and $p\in \R$, then $f$ has a well-defined ``derivative'' $Df_p\co
C_p(B)\to \R$ at the point $p$; this derivative is linear on each ray.

\begin{definition}[spherical join, cf. {\cite[I.5]{BrHa}}]\label{d:sphjoin}
Let $A$ and $B$ be metric spaces.
We may give the product $[0,\frac{\pi}{2}]\times A\times B$ a
pseudometric as follows.  Let $x_1 = (\phi_1,a_1,b_1)$ and $x_2 =
(\phi_2,a_2,b_2)$.  The distance $d(x_1,x_2)$ is given by
\begin{eqnarray*}
\cos \left(d(x_1,x_2)\right) &  = & 
\cos(\phi_1)\cos(\phi_2)\cos\left(d_\pi(a_1,a_2)\right) \\ & & +
\nonumber \sin(\phi_1)\sin(\phi_2)\cos\left(d_\pi(b_1,b_2)\right),
\end{eqnarray*}
where $d_\pi(p_1,p_2)$ means $\min\{\pi,d(p_1,p_2)\}$.
The \emph{spherical join} $A\join B$ of $A$ and $B$ is the canonical
metric space quotient of $[0,\frac{\pi}{2}]\times A\times B$ with the
above metric.
Points in $A\join B$ can be written as
triples $(\phi,a,b)$ with $\phi\in [0,\frac{\pi}{2}]$, $a\in
A$ and $b\in B$.  
If $\phi = 0$, the third coordinate can be
ignored, 
and if $\phi=\frac{\pi}{2}$, the second coordinate can be
ignored. 
\end{definition}

\begin{proposition}\cite{AB}\label{p:dirwarp}
Let $B$, $F$, and $f$ be as in Theorem \ref{t:AB}.  Let $p\in B$, and
let $\phi\in F$.  If $f(p)>0$, then there are isometries
\[ C_{(p,\phi)}(B\times_f F) \cong C_p(B)\times C_\phi(F) \]
and 
\[ \Sigma_{(p,\phi)}(B\times_f F)\cong
\Sigma_p(B)\join\Sigma_\phi(F).\]
If $f(p)=0$, there are isometries
\[ C_{(p,\phi)}(B\times_f F) \cong C_p(B)\times_{Df_p} F \]
and
\[ \Sigma_{(p,\phi)}(B\times_f F)\cong \Sigma_p(B)\times_{Df_p} F.\]
\end{proposition}

\subsection{Locally injective logarithms}
Let $X$ be a uniquely geodesic metric space and let $p\in X$.
In general, there is no \emph{exponential map} defined on the entire
tangent cone $C_p$.
However, there is a well-defined \emph{logarithm map} from $X$
to $C_p$ \cite{KL}.
\begin{definition}
Let $X$ be a uniquely geodesic metric space, and let 
$C_p=[0,\infty)\times_r\Sigma_p$ be the tangent
cone at a point $p\in X$.  If $x\in X$, then $\log_p(x)$ is the
equivalence class of $(d(x,p),[p,x])$ in $C_p$.
\end{definition}
If the metric is only locally
uniquely geodesic, then the logarithm map can still be defined in a
neighborhood of each point.  The existence of a local exponential map
is limited by the extent to which the logarithm is injective and
surjective.
\begin{definition}\label{d:linj}
The 
\emph{log-injectivity radius at $p$}, $\linj(p)$, is be the supremum of
those numbers $r$ for which $\log_p$ is well-defined and injective on the ball of
radius $r$ about $p$.  A neighborhood of $p$ on which $\log_p$ is
injective is called a \emph{log-injective neighborhood of $p$}.

Similarly, the \emph{log-surjectivity radius at $p$}, $\lsurj(p)$, is
the supremum of those numbers $r$ for which
$[0,r]\times_r\Sigma_p\subset C_p$ is contained in the image of
$\log_p$.
\end{definition}

In general, neither $\linj(p)$ nor $\lsurj(p)$ is positive.  Recall
that a space has the \emph{geodesic extension property} if any locally
geodesic segment terminating at $y$ can be extended to a longer
locally geodesic segment terminating at some $y'\neq y$.
A complete CAT$(0)$ space satisfying the geodesic extension property
has $\lsurj(p) = \infty$.
\begin{observation}\label{obs:exp}
If $R=\min\{\lsurj(p),\linj(p)\}$, then there is an
\emph{exponential map}
\[\exp_p\co [0,R)\times_r\Sigma_p\to X\]
which is a homeomorphism onto the open $R$--neighborhood of $p$.
More generally there is always an exponential map
\[\exp_p\co {\log_p}(U)\to U\subset X\]
for any log-injective neighborhood
$U$ of $p$, and this map is a homeomorphism onto its image.
\end{observation}

\begin{definition}
A subset $V$ of a uniquely geodesic space is 
\emph{convex}
if for any two points in $V$, the geodesic connecting them is
contained in $V$.  The subset $V$ is \emph{strictly convex} if it is
convex and if the frontier of $V$
contains no non-degenerate geodesic segment.
\end{definition}

\begin{lemma}\label{l:dconvex}
Let $X$ be a proper uniquely geodesic space satisfying the
geodesic extension property.
Suppose that $C$ is a compact 
strictly convex set in $X$ containing $p$ in its
interior, and suppose that $C$ is contained in a log-injective
neighborhood of $p$.  It follows that the frontier of
$C$ is homeomorphic to
$\Sigma_p$.
\end{lemma}
\begin{proof}
Since under the hypotheses both the space of directions $\Sigma_p$ and
the frontier of $C$ (which we denote here by $\partial C$)
are compact Hausdorff spaces, it suffices to find
a continuous bijection from one to the other.

Let $h\co \partial C\to \Sigma_p$ be the map which sends a point $x$ to the
equivalence class of $[x,p]$.  Since $p\notin \partial C$, this map is
well-defined.
Since $X$ is a proper uniquely geodesic space, geodesic segments vary
continuously with their endpoints \cite[I.3.13]{BrHa}.  It
follows that the map $h$ is continuous.

We next show that $h$ is surjective.
Let $\theta\in \Sigma_p$, and choose a geodesic segment beginning at
$p$ with direction $\theta$.  Since $X$ satisfies the geodesic
extension property, this geodesic segment can be
extended to a point $x$ not in $C$.  Some point on the geodesic
$[p,x]$ must be in the frontier of $C$, and so $\theta$ is in the
image of $C$.  

Finally we show that $h$ is injective.  Choose $\theta\in \Sigma_p$,
and suppose that $h(x)=h(y)=\theta$ for some $x, y$ in $\partial C$.
Without loss of generality, suppose that $d(x,p)\leq d(y,p)$.
Since $C$ is strictly convex, the geodesic from $x$ to $y$ must be
contained in $C$ but cannot be contained entirely in $\partial C$.
Thus there is some point $z\in [x,y]$ so that $z$ is in the interior
of $C$.  By the hypotheses and Observation \ref{obs:exp}
there is an open set 
$U\subset C_p= [0,\infty)\times_r\Sigma_p$ containing $\log_p(C)$ so
that the exponential map $\exp_p$ is a homeomorphism from $U$ to its
image.  The point $\log_p(z)$ is contained in some basic open subset
of $U$ of the form $(d(z,p)-\epsilon,d(z,p)+\epsilon)\times V$ where
$V$ is an open set in $\Sigma_p$ containing $\theta$.  
But since $C$ is convex, $\log_p(C)$ must also contain the union of
the segments beginning at $p$ and ending in
$(d(z,p)-\epsilon,d(z,p)+\epsilon)\times V$; in other words,
$\log_p(C)$ must contain $(0,d(z,p)+\epsilon)\times V$.  In
particular, $C$ contains an open neighborhood of any point on the
geodesic between $p$ and $z$, and thus $x$ must be an interior point
of $C$, which is a contradiction.
\end{proof}

\section{The metric construction and proof of Theorem
  \ref{t:filling}}\label{s:basic}
In this section we prove Theorem \ref{t:filling}
by constructing suitable warped product metrics on the ``partial cones''
$C(N_i,T_i)$ (see Definition \ref{def.cone})
so that those metrics are compatible with the hyperbolic
metric on $\overline M$.  The metric will be shown to be CAT$(0)$ by
applying Theorem \ref{t:AB} near the singular part, and by computing
the sectional curvatures in the Riemannian part.

\subsection{A model for the singular part}
We first make an observation about warped products, whose proof we
leave to the reader.
\begin{lemma}\label{l:tcs}
  Let $I\subseteq \R$ be connected, and let $f\co I\to [0,\infty)$,
  and let $F$ be a geodesic metric space.  Suppose that $f$ has a
  local minimum at $z\in I$.
  \begin{enumerate}
  \item $\{z\}\times F$ is a convex subset of $Z = I\times_f F$.
  \item Let $y\in I$, and $x\in F$.
  The shortest path from $(y,x)$ to ${z}\times F$ in $Z$ is
    $[y,z]\times \{x\}$.
  \end{enumerate}
\end{lemma}

We next give some applications of Theorem
\ref{t:AB}.
\begin{lemma}\label{l:step1}
  Let $E$ be complete and CAT$(k)$ for $k\in [-1,0]$.  The warped product
  \[ B = [0,\infty)\times_{\cosh(r)}E \]
  is CAT$(k)$, and $\{(0,e)\mid e\in E\}\subset B$ is convex.
\end{lemma}
\begin{proof}
  Since $\cosh(r)$ is $\mc{F}(-1)$ convex, it is $\mc{F}(k)$--convex,
  by \ref{convex}.\eqref{c-fort}. The zero set $\cosh^{-1}(0)$ is
  empty, and $\inf_r \cosh(r)=1$, so condition (1) of Theorem
  \ref{t:AB} is verified with $K = K_F = k$.  It follows that $B$ is
  CAT$(k)$.  Since the warping function $\cosh(r)$ has a minimum at
  $0$, Lemma \ref{l:tcs} implies that $\{0\}\times E$ is convex.
\end{proof}

\begin{proposition}\label{p:double}
  Suppose that $E$ is complete and  CAT$(k)$ for $k\in \{-1,0\}$, and
  suppose that $F$  is CAT$(1)$.  The space 
  \[W = [0,\infty)\times_{\cosh(r)}E\times_{\sinh(r)}F \]
  is CAT$(k)$, and $\{(0,e,-)\mid e\in E\}\subset W$ is convex.
\end{proposition}
\begin{proof}
  By Lemma \ref{l:tcs}, the set $Y=\{0\}\times E$ is a convex subset of
  $B = [0,\infty)\times_{\cosh(r)}E$.  On $B$, the function 
  $d(r,e) =  r$ is the distance to $Y$.
  By Lemma \ref{l:step1} and Useful fact \ref{convex}.\eqref{c-dist},
  the function $d(\cdot, Y)$ is convex on $B$

  Suppose $k = 0$.
  The function $\sinh(r)$ is the composition of an
  increasing convex function with a convex function, hence it is
  convex.  If
  $k=-1$, then $\sinh(r)$ is $\mc{F}(-1)$--convex by 
  \ref{convex}.\eqref{c-sinh}.  In either case, $\sinh(r)$ is
  $\mc{F}k$--convex, and so we can try to apply Theorem \ref{t:AB} to
  the warped product
  \[ W = B \times_{\sinh(r)} F.\]
  We must verify
  condition (2) of Theorem \ref{t:AB}, since $Y = \sinh^{-1}(0)$ is
  nonempty.  Let $b=(z,e)\in B\setminus Y$.  We can apply the second
  part of Lemma \ref{l:tcs} to see that a shortest unit speed geodesic
  $\sigma$ from $F$ to $b$ is of the form $\sigma(t) = (t,e)$, with
  domain $[0,z]$.  We thus have $(\sinh\circ \sigma)'(0^+) = \cosh(0)
  = 1 \geq 1$ as required.

  Finally, we show that 
  $Z = \{(0,e,-)\mid e\in E\}$ is convex.  If $\sigma\co[0,1]\to W$ is any
  rectifiable path with both endpoints in $Z$, and $\sigma(t) =
  (r(t),e(t),\theta(t))$, we
  note that $\bar{\sigma}(t) = (0,e(t),-)$ is
  a shorter path with the same endpoints.  It follows that any
  geodesic with both endpoints in $Z$ must lie entirely in $Z$.
\end{proof}

Using the convention that $\E^{0}$ is a point, the euclidean space
$\E^k$ is CAT$(0)$ for all $n$, and is CAT$(-1)$ for $k=0$ and $k=1$.
We thus have:
\begin{corollary}\label{c:singular}
  Let $T$ be a complete CAT$(1)$ space, and let $k$ be a nonnegative
  integer.  The warped product
  \[\tilde{F} = [0,\infty)\times_{\cosh(r)}\E^k\times_{\sinh(r)}T\]
  is complete and CAT$(0)$.  If $k\leq 1$, then $\tilde{F}$ is CAT$(-1)$.
\end{corollary}
\begin{lemma}\label{l:completion}
  Let $T$ be a flat manifold.  The warped product   
  \[\tilde{F} = [0,\infty)\times_{\cosh(r)}\E^k\times_{\sinh(r)}T\]
  is the metric completion of the Riemannian manifold
  \[D=(0,\infty)\times_{\cosh(r)}\E^k\times_{\sinh(r)}T\]
\end{lemma}
\begin{proof}
  The manifold $D$ is clearly dense in $\tilde{F}$, so it suffices to
  show that $D$ includes isometrically in $\tilde{F}$.  Let
  $x_1=(r_1,e_1,t_1)$ and $x_2=(r_2,e_2,t_2)$ be two points in $D$,
  and let $\epsilon>0$.
  Let $\gamma$ be a lipschitz path in the product
  $[0,\infty)\times \E^k\times T$ nearly realizing the distance
   between $x_1$ and $x_2$ in $\tilde{F}$,
  so that the length in $\tilde{F}$ of $\gamma$ is at most
  $d_{\tilde{F}}(x_1,x_2)+\epsilon$.  
  We will replace $\gamma$ by a path $\gamma'$ in $D$ so that the
  length of $\gamma'$ exceeds the length of $\gamma$ by at most
  $\delta(\epsilon)$, with $\lim_{\epsilon\to
    0^+}\delta(\epsilon)=0$.  Letting $\epsilon$ tend to zero, the
  lemma will follow.

  For $t$ in the domain of
  $\gamma$, we have $\gamma(t) = (r(t),e(t),\theta(t))$.  If $r(t)$ is
  positive for all $t$, then $\gamma$ stays inside $D$, and there is
  nothing to prove.  We therefore assume that $r(t)=0$ for $t$ in a
  (possibly degenerate) interval $[t_1,t_2]$.
  By Proposition \ref{p:double},  $\{(0,e,-)\mid e\in \E^k\}$ is
  convex, so 
  we have $r(t)>0$ for any $t\notin [t_1,t_2]$.

  The
  coordinates of $\gamma$ are uniformly continuous in $t$, so we can
  find small positive $\alpha_1$ and $\alpha_2$ so that 
  $r(t_1-\alpha_1)=r(t_2+\alpha_2)$ and all the
  following are satisfied:
  \begin{eqnarray*}
    \max\{d(e(t_1-\alpha_1),e(t_1)),d(e(t_2+\alpha_2),e(t_2))\}& < &\epsilon\\
    \max\{d(\theta(t_1-\alpha_1),\theta(t_1)),d(\theta(t_2+\alpha_2),\theta(t_2))\}&
    < &\epsilon\\
    \max\{\cosh(r(t_1-\alpha_1))-1,\sinh(r(t_1-\alpha_1))\} & <
    &\epsilon
  \end{eqnarray*}
  Let $e'\co [t_1-\alpha_1,t_2+\alpha_2]\to \E^k$ be a constant-speed
  geodesic from $e(t_1-\alpha_1)$ to $e(t_2+\alpha_2)$, and let
  $\theta'\co[t_1-\alpha_1,t_2+\alpha_2] $ be a constant-speed
  geodesic from $\theta(t_1-\alpha_1)$ to $\theta(t_2+\alpha_2)$. Let
  $\gamma'$ be given by
\[ \gamma'(t) = 
\begin{cases}
  \gamma(t) & t<t_1-\alpha_1\\
  (r(t_1-\alpha_1),e'(t),\theta'(t)) & t_1-\alpha_1\leq t\leq
  t_2+\alpha_2\\
  \gamma(t) & t>t_2+\alpha_2
\end{cases}
  \]
  As the reader may check,
  the difference between the length of $\gamma$ and the length of
  $\gamma'$ is at most
  $\epsilon(d(\theta(t_1),\theta(t_2)+d(e(t_1),e(t_2))+\epsilon+2\epsilon^2$.
  Letting $\epsilon$ tend to zero, we have established the lemma.
\end{proof}

If $T$ is a Riemannian
manifold, the warped product from
Corollary \ref{c:singular} is a Riemannian manifold
in a neighborhood of any $p\notin\{0\}\times \E^k \times T$.
It follows that the space of
directions at $p$ is a sphere.  
The next lemma describes the space
of directions at a non-manifold point.

\begin{lemma}\label{l:dirjoin}
Let $\tilde{F}$ be as in Corollary \ref{c:singular}.
The space of directions $\Sigma_{(0,e,-)}(\tilde{F})$ at a point $(0,e,-)$ of
$\tilde{F}$ is isometric to the spherical join $S^{k-1}\join T$.  
\end{lemma}
\begin{proof}
Proposition \ref{p:dirwarp} can be applied twice, as follows.  Let
$B = [0,\infty)\times_{\sinh(r)}T$, and apply the second part of
Proposition \ref{p:dirwarp} to deduce 
\[\Sigma_{(0,-)}(B)\cong T. \]
Next, since $\tilde{F} \cong B\times_{\cosh{r}}\E^k$, we can apply the first
part of Proposition \ref{p:dirwarp} to deduce that
\[\Sigma_{(0,e,-)}(\tilde{F})\cong \Sigma_{(0,-)}(B)\join \Sigma_e(E).\] 
\end{proof}
\begin{lemma}\label{l:gep}
  Suppose $T$ is a complete flat Riemannian 
  manifold with injectivity radius bigger than
  $\pi$.  Then $\tilde{F} =
  [0,\infty)\times_{\cosh(r)}\E^k\times_{\sinh(r)}T$ has the geodesic
    extension property.
\end{lemma}
\begin{proof}
  Since $\tilde{F}$ is CAT$(0)$ by Corollary \ref{c:singular},
  geodesics are the same as local geodesics.
  Away from $V=\{(0,e,-)\mid
  e\in \E^k\}$, the space $\tilde{F}$ is Riemannian, so geodesics can
  be extended in $\tilde{F}\setminus V$. 

  At a point $p$ of $V$, Lemma \ref{l:dirjoin} implies that the space of
  directions $\Sigma_p$ is isometric to $S^{k-1}\join T$.  
  Suppose $\sigma$ is a geodesic segment terminating at $p$, and let
  $[\sigma]=(\phi,\alpha,\theta)\in \Sigma_p$ be the direction of
  $\sigma$, where $\phi\in [0,\frac{\pi}{2}]$, $\alpha\in S^{k-1}$,
  and $\theta\in T$.  Since $T$ has injectivity radius bigger than
  $\pi$, there is some $\theta'$ with $d_T(\theta,\theta')=\pi$.  It
  is straightforward to check that the Alexandrov angle between
  $[\sigma]$ and $(\phi,-\alpha,\theta')$ is $\pi$.  Letting $\sigma'$
  be any geodesic segment starting at $p$, with direction
  $(\phi,-\alpha,\theta')$, we see that $\sigma'$ geodesically extends $\sigma$
  past $p$.
\end{proof}

We next argue that the space $\tilde{F}$ described in Corollary
\ref{c:singular} has log-injective neighborhoods at the singular
points, in the sense of Definition \ref{d:linj}.  It suffices to show that geodesic segments emanating from a
point in the singular set can only make an angle of zero if they coincide on an
initial subsegment.  We first find the direction of an arbitrary
geodesic segment emanating from a point in the singular set.
\begin{lemma}\label{l:direction}
Let $T$ be a flat manifold of injectivity radius larger than $\pi$.
Let $p_0$ be a point in 
$E = \{0\}\times \E^k\subset [0,\infty)\times_{\cosh(r)}\E^k\times_{\sinh(r)}T$, 
and let $p_1$ be some
other point of $\tilde{F}$; that is,
$p_0=(0,a_0,-)$ and $p_1 = (t_1,a_1,\theta_1)$ for some $t_1\geq
0$, $a_0$ and $a_1\in E$ and $\theta_1\in T$.  
Let $\sigma$ be a geodesic from $p_0$ to $p_1$.
If $a_1=a_0$, then the direction
of $\sigma$ at $p_0$
(in the spherical join coordinates of Definition \ref{d:sphjoin}) is
$(\frac{\pi}{2},-,\theta_1)$. Otherwise let $\alpha$ be the
direction of a geodesic in $E$ from $a_0$ to $a_1$;
the direction of $\sigma$
is $(\phi,\alpha,\theta_1)$ where
\[\tan(\phi)=\frac{\tanh(t_1)}{\sinh(|a_1-a_0|)}.\] 
\end{lemma}
\begin{proof}
If $a_1 = a_0$ the Lemma is obvious.
Otherwise, we identify $\R$ with a geodesic $\gamma$ in $E$ passing through $a_0$ and $a_1$,
setting $a_0 = 0$.  The warped product $W =
[0,\infty)\times_{u\cosh(t)}\R\subset \tilde{F}$ contains the geodesic from
$p_0$ to $p_1$.

The map $h(t,a) = e^{ua}(\tanh t+i\mathrm{sech}\ t)$ takes $W$
isometrically onto the hyperbolic half plane in the upper half-space
model consisting of points with nonnegative real part. The map $h$ takes $p_0$
to $i$ and $p_1$ to some point of modulus at least $1$.  The geodesic
$\gamma$ is sent to the positive imaginary axis.
It is an exercise in
hyperbolic geometry to verify that if the geodesic from $h(p_0)$
to $h(p_1)$ makes an angle of $\phi$ with $h(\gamma)$, then
$\tan(\phi)=\frac{\tanh(t_1)}{\sinh(|a_1-a_0|)}$.
The lemma follows.
\end{proof}

\begin{corollary}\label{c:loginj}
Let $p=(0,a,-)$ lie in $E\subset \tilde{F}$.  Any open neighborhood of $p$ is
log-injective at $p$.
\end{corollary}
\begin{proof}
  Let $p_1 = (t_1,a_1,\theta_1)$ and $p_2 = (t_2,a_2,\theta_2)$ be
  points of $\tilde{F} =
  [0,\infty)\times_{\cosh(r)}\E^k\times_{\sinh(r)}T$, and suppose that
  $\log(p_1)=\log(p_2)$ in the tangent cone at $p$.
  Let $\gamma_i$ be the unique geodesic from $p$ to $p_i$, for $i\in
  \{1,2\}$.  Since $\log(p_1) = \log(p_2)$, the geodesic segments
  $\gamma_1$ and $\gamma_2$ have the same length and the same
  direction $(\phi,\alpha,\theta)$ as described in Lemma
  \ref{l:direction}.  In particular, the segments $[a,a_1]$ and
  $[a,a_2]$ in $E$ must have the same direction at $a$.  The
  log-injectivity radius at any point in Euclidean space is infinite, so these
  geodesic segments
  are subsets of a single line in $E$. 
  As in the proof of Lemma \ref{l:direction}, the geodesics $\gamma_1$
  and $\gamma_2$ both lie in a subset of $\tilde{F}$ isometric to a
  hyperbolic halfplane.  Since a hyperbolic half-plane has infinite
  log-injectivity radius, $\gamma_1$ and $\gamma_2$ must coincide.
\end{proof}
\begin{remark}
  If one is merely interested in nonpositive curvature, and not
  log-injectivity, there is more flexibility in the choice of warping
  functions.  Suppose that $T$ is a flat manifold.  It is not hard to
  show that the space
\[ \tilde{F}' = [0,\infty)\times_{e^r}\E^k\times_{\sinh(r)}T \]
  is CAT$(0)$.  On the other hand, $\tilde{F}'$ does not have
  log-injective neighborhoods at points $(0,e,-)$.
\end{remark}

\subsection{Curvatures in Riemannian warped products}
The following proposition can be proved by a simple if tedious computation involving
Christoffel symbols.  See \cite{BlHo} for a proof in case both flat
factors are $1$--dimensional.
\begin{proposition}\label{p:sectional}
  Let $A_1$ and $A_2$ be flat manifolds, and let $I\subseteq \R$.  Let
  $f_1$, $f_2$ be positive smooth functions on $I$.
  The
  warped product manifold 
  \[ W = I\times_{f_1} A_1 \times_{f_2} A_2 \]
  has sectional curvatures at $(t,a_1,a_2)\in W$ which are convex combinations
  of the following functions, evaluated at $t$:
\[ -\frac{f_1''}{f_1}, -\frac{f_2''}{f_2}, -\frac{(f_1')^2}{f_1^2},
-\frac{(f_2')^2}{f_2^2},\mbox{ and } -\frac{f_1'f_2'}{f_1f_2}\]
  Let $i\in \{1,2\}$.
  If $\dim(A_i)=1$, then the term
  $-\frac{(f_i')^2}{f_i^2}$ can be ignored.  If $\dim(A_i) = 0$, then
  all terms involving $f_i$ and its derivatives can be ignored.
\end{proposition}
It follows from this proposition that if $f_1$ and $f_2$ are positive,
convex, and increasing,
the warped product
$I\times_{f_1}A\times_{f_2}B$ is nonpositively curved.  If $f_1$,
$f_2$, and their first and second derivatives are all bounded between
two positive numbers $\epsilon<1$ and $R>1$, then the curvature of
$I\times_{f_1}A\times_{f_2}B$ is bounded between
$-\frac{R^2}{\epsilon^2}$, and $-\frac{\epsilon^2}{R^2}$.

\subsection{Gluing functions}

We will construct smooth convex increasing 
functions $f$ and $g$ on $[0,1+\lambda]$
interpolating between the exponential function $e^{r-1}$ near
$r=\lambda$, and the functions $\sinh(r)$ and $\cosh(r)$ near $r=0$.
We will apply the following result of Agol, proved in \cite{Ag}.

\begin{lemma}\cite[Lemma 2.5]{Ag}\label{l:agol}
Suppose
\begin{equation*}
a(r)=
    \begin{cases}
        b(r)& r<R ,\\
        c(r) & r\geq R,
    \end{cases}
\end{equation*}
where $b(r)$ and $c(r)$ are $C^\infty$ on $(-\infty,\infty)$, and
$b(R)=c(R)$, $b'(R)=c'(R)$.
Then we may find $C^{\infty}$ functions $a_\epsilon$  on $(-\infty,\infty)$
for $\epsilon>0$ such that
\begin{enumerate}
\item
there is a $\delta(\epsilon)>0$ such that $\underset{\epsilon\to 0^+}{\lim}
\delta(\epsilon)=0$ and $a_\epsilon(r)=b(r)$ for $r\leq R-\delta(e)$, $a_\epsilon(r)=c(r)$
for $r\geq R$, and
\item
$$\min\{b''(R),c''(R)\}=\underset{\epsilon\to 0^+}{\lim} \underset{R-\delta(\epsilon)\leq r \leq R}{\inf} a_\epsilon''(r)$$
$$\leq \underset{\epsilon\to 0^+}{\lim} \underset{R-\delta(\epsilon)\leq r \leq R}{\sup} a_\epsilon''(r)
= \max\{b''(R),c''(R)\}.$$
\end{enumerate}
\end{lemma}

Agol's lemma allows us to smoothly interpolate between functions which
agree up to first order at some point.  We wish to interpolate between
pairs of 
functions (either $\sinh(r)$ and $e^{r-1}$  or $\cosh(r)$ and
$e^{r-1}$) which do not agree anywhere up to first order.  To solve
this difficulty, we introduce a third function which interpolates
between the pair:
\begin{lemma}\label{l:interp}
  Let $l_1(x) = ax+b$, $l_2(x)=cx+d$.  Suppose $a < c$, and let $[A,B]$
  be an interval containing $\frac{b-d}{c-a}$ (where the graphs of $l_1$
  and $l_2$ meet). There is a $k>0$ and a smooth function $\varepsilon$ on $[A,B]$
  so that:
\begin{enumerate}
\item $\varepsilon'(A) = a$,
\item $\varepsilon'(B) = c$, and
\item $\varepsilon''(t) > k$ on $[A,B]$.
\end{enumerate}
\end{lemma}
\begin{proof}
  One way to do this is to start with the circle 
  \[S = \{(x,y)\mid(x-1)^2+(y-1)^2=1\},\]
  and let $M\co \R^2\to \R^2$ be the unique affine map taking $(0,1)$
  to $(A,l_1(A))$, $(1,0)$ to $(B,l_2(B))$, and $(0,0)$ to the point
  of
  intersection of the graphs of $l_1$ and $l_2$.  The resulting
  ellipse $M(S)$ is tangent to the graph of
  $l_1$ at $(A,l_1(A))$, and to the graph of $l_2$ at $(B,l_2(B))$.
  The part of the ellipse between these two points and closest to the
  union of the graphs of $l_1$ and $l_2$ is the graph of a function
  $\varepsilon$ satisfying the requirements of the lemma.  See Figure
  \ref{fig:affine} for an illustration.
\begin{figure}[htbp]
\begin{center}
\input{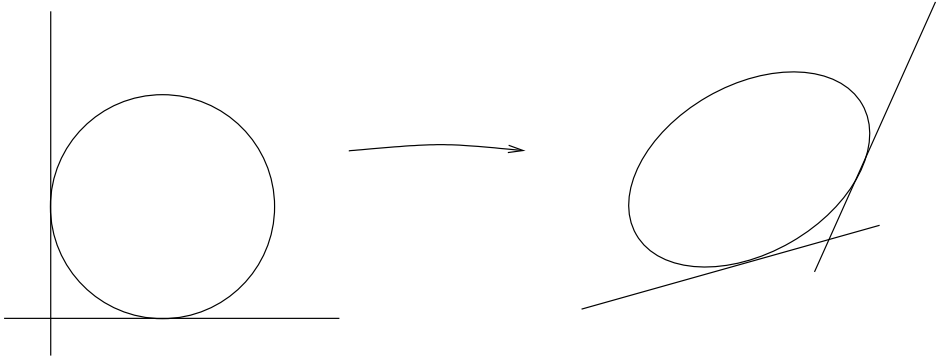}\label{fig:affine}
\caption{The curve on the right closest to the lines is a $C^1$
  interpolation between the lines.}
\end{center}
\end{figure}

\end{proof}
\begin{proposition}\label{p:warpfg}
  For any $\lambda>0$ there is a $\delta>0$, $k>0$, and a pair of smooth
  functions $f\co \R_+\to \R_+$ and $g\co \R_+\to \R_+$ satisfying:
\begin{enumerate}
\item for $r<\delta$, $f(r) = \sinh(r)$ and $g(r) = \cosh(r)$,
\item for $r>1+\frac{\lambda}{2}$, $f(r) = g(r) = e^{r-1}$, and 
\item for $r>0$, $f''(r)>0$, and for $r\in \R_+$, $g''(r)>k$.
\end{enumerate}
\end{proposition}
\begin{proof}
  We first define $C^1$ functions $f_0$ and $g_0$
  with the above properties.  We
  then invoke Agol's gluing lemma (Lemma \ref{l:agol}) to obtain
  smooth functions.  

  Let $l_2$ be the equation of the tangent line to $h(r) = e^{r-1}$ at
  $r=1+\frac{\lambda}{2}$, and note that $l_2(\frac{\lambda}{2})=0$.
  Since
  $h(1+\frac{\lambda}{2}) =e^{\frac{\lambda}{2}}>1+\frac{\lambda}{2}$, the line $l_2$
  intersects the tangent lines at $0$ to both $\cosh(r)$ and
  $\sinh(r)$ somewhere in the interval $(0,1+\frac{\lambda}{2})$.  By
  continuity, there is some $\delta_0>0$ so that the tangent lines to
  $\cosh(r)$ and $\sinh(r)$ at $r=\delta_0$ still hit $l_2$ somewhere
  in the interval $(0,1+\frac{\lambda}{2})$.

  In particular, if $l_1$ is the tangent 
  line to $\sinh(r)$ at $r=\delta_0$, 
  we may apply Lemma
  \ref{l:interp} to obtain a function $\varepsilon_f$ on
  $(\delta_0,1+\frac{\lambda}{2})$ with $\varepsilon_f'(\delta_0) =
  \cosh(\delta_0)$, $\varepsilon_f'(1+\frac{\lambda}{2}) = h'(1+\frac{\lambda}{2})$, and
  $\varepsilon_f''>k_f$ everywhere, for some $k_f>0$.  We define
  \[ f_0(r) = 
  \begin{cases}
    \sinh(r) & r\leq \delta_0\\
    \varepsilon_f(r) & \delta_0<r<1+\frac{\lambda}{2}\\
    e^{r-1} & r\geq 1+\frac{\lambda}{2}
  \end{cases}  
  \]
  (see Figure \ref{f:interp}).
  \begin{figure}[htbp]
    \begin{center}
      \input{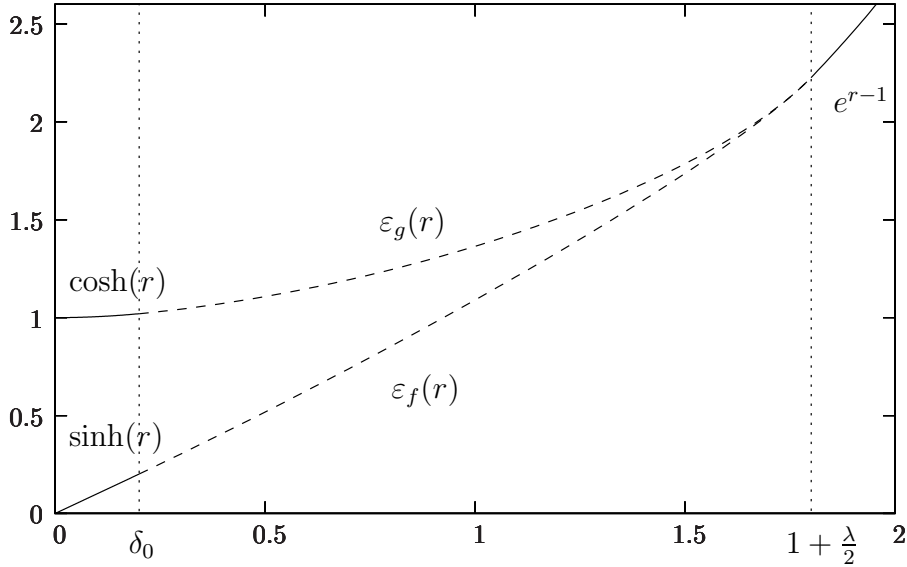}
      \caption{The bottom curve is $f_0$, and the top curve is $g_0$,
        for $\lambda=1.6$ and $\delta_0=0.2$.}
      \label{f:interp}
    \end{center}
  \end{figure}
  Choosing some positive $\epsilon<\min\{\delta_0,\frac{\lambda}{2}\}$, we may 
  twice apply Lemma \ref{l:agol} (at $r=\delta_0$ and at $r=1+\frac{\lambda}{2}$)
  to
  obtain a smooth function $f\co \R_+\to\R_+$ satisfying:
  \begin{enumerate}
  \item $f(r) = f_0(r)$ outside the intervals
    $(\delta_0-\epsilon,\delta_0)$ and $(1+\frac{\lambda}{2}-\epsilon,1+\frac{\lambda}{2})$,
  \item $f''(r)> 0.9 \min\{\sinh(\delta_0),
    \varepsilon_f''(\delta_0)\}$ on $(\delta_0-\epsilon,\delta_0)$,
    and
  \item $f''(r)> 0.9 \min\{\varepsilon_f''(1+\frac{\lambda}{2}),e^{\frac{\lambda}{2}}\}$ on
    $(1+\frac{\lambda}{2}-\epsilon, 1+\frac{\lambda}{2})$.
  \end{enumerate}
  (Here $0.9$ can be replaced with any number less than $1$.)

  A similar argument can be used to construct $g$.   Lemma
  \ref{l:interp}  can be used to
  construct a function $\varepsilon_g$ on the interval
  $(\delta_0,1+\frac{\lambda}{2})$ with $\varepsilon_g'(\delta_0) =
  \sinh(\delta_0)$, $\varepsilon_g'(1+\frac{\lambda}{2}) = e^{\frac{\lambda}{2}}$, and
  $\varepsilon_g''>k_g$ everywhere.  We can thus assemble a $C^1$
  function:
  \[ g_0(r) = 
  \begin{cases}
    \cosh(r) & r\leq \delta_0\\
    \varepsilon_g(r) & \delta_0<r<1+\frac{\lambda}{2}\\
    e^{r-1} & r\geq 1+\frac{\lambda}{2}.
  \end{cases} \]
  Applying Lemma \ref{l:agol} twice to $g_0$ and using the same
  $\epsilon$ as above
  we obtain a function $g$ satisfying:
  \begin{enumerate}
  \item $g(r) = g_0(r)$ outside the intervals
    $(\delta_0-\epsilon,\delta_0)$ and $(1+\frac{\lambda}{2}-\epsilon,1+\frac{\lambda}{2})$,
  \item $g''(r)> 0.9 \min\{\cosh(\delta_0),
    \varepsilon_g''(\delta_0)\}$ on $(\delta_0-\epsilon,\delta_0)$,
    and
  \item $g''(r)> 0.9 \min\{\varepsilon_g''(1+\frac{\lambda}{2}),e^{1+\frac{\lambda}{2}}\}$ on
    $(1+\frac{\lambda}{2}-\epsilon, 1+\frac{\lambda}{2})$.
  \end{enumerate}
  The functions $f$ and $g$ satisfy the conclusion of the proposition for
  $\delta = \delta_0 - \epsilon$, and $k = \min\{0.9 k_g,1\}$.
\end{proof}
\begin{theorem}\label{t:nonpos}
  Let $f$ and $g$ be as in Proposition \ref{p:warpfg},  and let $T$ be
  flat manifold with injectivity radius bigger than $\pi$.  The warped product
  \[ Z = [0,1+\lambda]\times_{g(r)}\E^k\times_{f(r)}T \]
  is Riemannian away from $\{0\}\times \E^k\times T$.  The space $Z$
  is complete and CAT$(0)$, and is CAT$(-\kappa)$ for some $\kappa>0$ if $k\leq 1$.
\end{theorem}
\begin{proof}
  Since $g(r)$ and $f(r)$ are positive and smooth for $r>0$, the space
  $Z$ is Riemannian away from $\Xi = \{0\}\times \E^k\times T$.  Since
  $f(0)=0$, the space $Z$ is simply connected, so it suffices to check
  the curvature locally.  Near the singular set $\Xi$, the $Z$ is locally
  isometric to a subset of the space $\tilde{F}$ described in
  Corollary \ref{c:singular}.  It follows that $Z$ is locally CAT$(0)$
  (and locally CAT$(-1)$ if $k\leq 1$)
  in a $\delta$--neighborhood of $\Xi$.

  Away from $\Xi$, we may estimate derivatives, and
  apply Proposition \ref{p:sectional}.  Since $f$, $g$, and their
  first and second derivatives are all positive and continuous on the
  interval $[\delta,1+\lambda]$, there is a positive lower bound
  $\kappa$ valid
  for all 
  the quantities
 \[ \frac{f''}{f}, \frac{g''}{g}, \frac{(f')^2}{f^2},
 \frac{(g')^2}{g^2},\mbox{ and } \frac{f'g'}{fg} \]
  on this interval.  Proposition \ref{p:sectional} implies that the
  sectional curvatures are thus bounded above by $-\kappa$, away from
  a $\delta$--neighborhood of $\Xi$.

  Putting the local pictures together, $Z$ is locally (and hence
  globally) CAT$(-\kappa)$ if $k\leq 1$, and CAT$(0)$ otherwise.
\end{proof}

\subsection{Nonpositively curved metrics on the partial cones}
If $M$ is a hyperbolic manifold, and $E\subset M$ is a closed horospherical
neighborhood of a toral cusp, then $E$ is isometric to a warped
product
\[ E = (-\infty,0]\times_{e^r}N, \]
where $N = \partial E$ with the induced flat Riemannian metric.  
(To see this, recall that if we identify $\H^{n+1}$ with
$\R\times_{e^r}\E^n$, then sets of the form $\{R\}\times \E^n$ are
concentric horospheres.  Any discrete group $P$
of parabolic isometries of
$\H^{n+1}$ can be realized as a group preserving these horospheres.
The quotient by the action is then of the form $\R\times_{e^r}F$ for
some flat manifold $F$.  If $P$ is a maximal parabolic subgroup of a
torsion-free
lattice $\Gamma$, some subset of the form $(-\infty,t]\times_{e^r}F$ embeds
in $\H^{n+1}/\Gamma$
By rescaling $F$ to another flat manifold $F'$, this subset looks like
$(-\infty,0]\times_{e^r}F'$.)

If
$E$ is embedded in $M$, then there is some $\lambda>0$ so that
$E_\lambda$, the closed
$\lambda$--neighborhood of $E$, is still embedded.  We have
\[ E_\lambda = (-\infty,\lambda]\times_{e^r}N. \]
It is convenient to reparameterize $E_\lambda$ as
\[ E_\lambda = (-\infty,1+\lambda]\times_{e^{r-1}}N. \]
Corresponding to any geodesic submanifold $T$ there is a cover $\tilde{N}$
of $N$, isometric to $\E^k\times T$, where $k$ is the codimension of
$T$ in $N$.  Likewise, there is a cover $\tilde{E}_\lambda$ of
$E_\lambda$ corresponding to $T$, so that
\begin{eqnarray*}
\tilde{E}_\lambda & = & (-\infty,1+\lambda]\times_{e^{r-1}}\tilde{N}\\
& = & (-\infty,1+\lambda]\times_{e^{r-1}}\E^k\times_{e^{r-1}} T .
\end{eqnarray*}
The group $Q = \pi_1(N)/\pi_1(T)$ acts, preserving the product
structure, by isometries on 
$\tilde{E}_\lambda$, with quotient $E_\lambda$.  If the warping
functions are replaced by any other warping functions,
$\pi_1(N)/\pi_1(T)$ will still act by isometries.  We will use the
functions $f$ and $g$ defined in Proposition \ref{p:warpfg} to define
a new metric, first on
\[ \tilde{W} = [0,1+\lambda]\times \E^k\times T \]
and then on the space $W=\tilde{W}/Q$, which is homeomorphic to $C(N,T)$.
\begin{theorem}\label{t:metricplug}
  Let $M$ be a hyperbolic $(n+1)$--manifold, with an embedded toral cusp 
  \[ E \cong (-\infty,0]\times_{e^r}N, \]
  and suppose $T\subseteq N$ is a totally geodesic
  $(n-k)$--dimensional torus with
  injectivity radius greater than $\pi$.  Choose $\lambda>0$ so that
  the $\lambda$--neighborhood of $E$ is still embedded in $M$, with
  boundary $N_\lambda$.  
  
  There is then a nonpositively curved metric on the partial cone $C(N,T)$ so that the
  following hold:
  \begin{enumerate}
  \item\label{mplug1} $C(N,T)$ is Riemannian away from its core.
  \item\label{mplug2} There is a Riemannian isometry between the
    $\frac{\lambda}{2}$--neighborhood of $\partial C(N,T)$, and the
    $\frac{\lambda}{2}$--neighborhood of $N_\lambda$ in $C_\lambda$,
    which takes $T\subset C(N,T)$ to a torus $T_\lambda\subset
    C_\lambda$, isotopic in $C_\lambda$ to $T$.
  \item\label{mplug3} If $k\leq 1$, then $C(N,T)$ is locally CAT$(-\kappa)$ for some
    $\kappa>0$. 
  \end{enumerate}
\end{theorem}
\begin{proof}
  Corresponding to $T\subseteq N$ there is a cover $\tilde{N}\to N$ so
  that $\tilde{N}$ is isometric to $\E^k\times N$.  Likewise, there is
  a corresponding cover of $E_\lambda$,
  which we parametrize as
  \[\tilde{E}_\lambda =
  (-\infty,1+\lambda]\times_{e^{r-1}}\E^k\times_{e^{r-1}}T.\]
  Let $f$ and $g$ be the functions from Proposition \ref{p:warpfg},
  and let
  \[ \tilde{W} = [0,1+\lambda]\times_{g(r)}\E^k\times_{f(r)}T \]
  By Theorem \ref{t:nonpos}, $\tilde{W}$ is CAT$(0)$, and is
  CAT$(-\kappa)$ for some $\kappa>0$ if $k\leq 1$.

  The group $Q = \pi_1(N)/\pi_1(T)$ acts freely and properly
  discontinuously by isometries on
  $\tilde{W}$, with quotient $W$ homeomorphic to $C(N,T)$.  Thus the
  quotient is locally nonpositively curved, locally CAT$(-\kappa)$ if
  $k\leq 1$, and Riemannian away from the core 
  \[\{(0,e,-)\mid e\in \E^k\}/Q. \]

  Note that if we write
  $\partial \tilde{W}$ for $\{1+\lambda\}\times \E^k\times T$, then there
  are $\frac{\lambda}{2}$--neighborhoods of $\partial\tilde{W}$ and
  $\partial\tilde{E}_\lambda$ which are canonically isometric.  This
  isometry descends to an isometry between
  $\frac{\lambda}{2}$--neighborhoods of $\partial W$ and $\partial
  E_\lambda$, which takes $T\subset W$ to $T_\lambda$ as
  described in the statement of the theorem.
\end{proof}
\begin{remark}
Since we obtain a compact space with the same metric near the boundary
as a cusp
by this operation,
this is sometimes called cusp closing (cf.\cite{Sc,An,HuSc}).
The difference in our setting
is that we are not generally closing the cusp ``as a manifold'' or
even as an
orbifold.
Rather, we show how to close a cusp ``as a pseudomanifold'',
with a certain curvature control.

A reverse operation, cusp opening or ``drilling'', has also been
studied. (See for example
\cite{Ag,HoKe} for applications to hyperbolic $3$--manifolds.)
For example, in \cite{Fu} a cusp is produced under 
a certain curvature control after removing 
a totally geodesic codimension two submanifold in a closed
hyperbolic manifold.  This construction also used doubly warped
products and was generalized in \cite{AbSc} (cf. \cite{Be1,Be2}).
\end{remark}

\subsection{Proof of Theorem \ref{t:filling} and Proposition \ref{p:hyper}}
Let $n\geq 2$.
Suppose $M$ is a finite volume hyperbolic $(n+1)$--manifold, 
with
horospherical cusps $E_1,\ldots,E_m$ embedded in $M$.  Suppose each of those cusps is
isometric to $(-\infty,0]\times N_i$ for a flat torus $N_i$.  
Choose $\lambda>0$ so that the $\lambda$--neighborhood of
$\bigcup_{i=1}^nE_i$ is still embedded in $M$.
For each
$i$, let $T_i$ be a totally geodesic torus in $N_i$ with
injectivity radius larger than $\pi$, so that the space
$M(T_1,\ldots,T_m)$ from Definition \ref{d:filling} is a
$2\pi$--filling of $M$. 

We will describe a metric on
$M(T_1,\ldots,T_m)$ satisfying the conclusions of Theorem
\ref{t:filling} and \ref{p:hyper}.

Fix $i\in \{1,\ldots,m\}$.  The $\lambda$--neighborhood of the cusp
$E_i$ is embedded, and disjoint from the $\lambda$--neighborhood of
all the other cusps.  Give $C(N_i,T_i)$ the nonpositively curved
metric from Theorem \ref{t:metricplug}.

For $t\in [0,\lambda]$, let $M_t\subset M$ be the closure of the
complement of the union of
the $t$--neighborhoods of the cusps $E_1,\ldots,E_m$.  Thus $M_0$ is
the compact manifold $\overline{M}$ from Definition \ref{d:filling}.  By
our choice of $\lambda$, each $M_t$ is an embedded submanifold with
boundary, and each $M_t$ is isotopic to $M_0$.
Let $D_i$ be an open $\frac{\lambda}{2}$--neighborhood of the
boundary of $C(N_i,T_i)$.
According to Theorem \ref{t:metricplug}.\eqref{mplug2}, there is an isometry
\[ \psi_i\co D_i\longrightarrow M_{\frac{\lambda}{2}}\setminus M_{\lambda}  \]
for each $i$, so that the space
\[Y = M_{\frac{\lambda}{2}}\bigcup_{\psi_1\sqcup\cdots\sqcup\psi_m}(C(N_1,T_1)\sqcup\cdots\sqcup C(N_m,T_m))\]
is homeomorphic to the filling $M(T_1,\ldots,T_m)$ by a homeomorphism
taking $M_{\frac{\lambda}{2}}\subset Y$ isometrically to $M_{\frac{\lambda}{2}}\subset
M(T_1,\ldots,T_m)$.  It is clear that
Proposition \ref{p:hyper} holds for this metric on
$M(T_1,\ldots,T_m)$, and $\overline{M}'=M_{\frac{\lambda}{2}}$.

Conclusion \eqref{fill1} of Theorem \ref{t:filling} holds by
construction; indeed, every non-singular point of $Y$ has a neighborhood
which is either isometric to a neighborhood in $M$ or to a
neighborhood in the non-singular part of some $C(N_1,T_1)$, endowed
with the metric from Theorem \ref{t:metricplug}.  This metric is
Riemannian by Theorem \ref{t:metricplug}.\eqref{mplug1}.  The fact
that the metric on $Y$ is the completion of the Riemannian metric on
the non-singular part can be deduced from Lemma \ref{l:completion}.

Every point in $Y$ has a neighborhood isometric either to a neighborhood of a point
in $M$, which has curvature bounded above by $-1$ by assumption, or in $C(N_1,T_1)$,
which is nonpositively curved by Theorem \ref{t:metricplug}.
Conclusion \eqref{fill2} of Theorem \ref{t:filling} is established.

If for every $i$, the torus $T_i$ has codimension at most $1$ in
$N_i$, then the spaces $C(N_i,T_i)$ have
curvature bounded above by $-\kappa$ for some $\kappa>0$, by Theorem
\ref{t:metricplug}.\eqref{mplug3}.  Conclusion \eqref{fill3} of
Theorem \ref{t:filling} follows. 

\subsection{Isolated flats}
Let $M(T_1,\ldots,T_m)$ be a $2\pi$--filling of the hyperbolic
$(n+1)$--manifold $M$.
In this section we concentrate on the case in which, for at least one
$i\in \{1,\ldots,m\}$, the torus $T_i$ has dimension at most $n-2$.
In this case the metric constructed in Theorem \ref{t:filling} is
nonpositively curved but not locally CAT$(-\kappa)$ for any positive
$\kappa$.  This is because the filling core $V_i$ of $C(N_i,T_i)$ is an
isometrically embedded $k$--torus, where $k = n-\dim(T_i)>1$.

The metric described in Theorem \ref{t:filling} lifts to a metric on 
the universal cover $X$ of $M(T_1,\ldots,T_m)$.
\begin{definition}
A convex
subset of a geodesic metric space which is isometric 
to $\E^d$ is called a \emph{flat}, if $d\geq 2$.
\end{definition}

The preimage of $V_i$ in $X$ is a union of
$k$--dimensional flats.  Conversely, every flat in $X$ is in the
preimage of some filling core, since the metric on $M(T_1,\ldots,T_m)$ 
is Riemannian and
negatively curved away from the cores, by Theorem
\ref{t:filling}.\eqref{fill1}. 

In \cite{HK}, Hruska and Kleiner give a definition of isolated flats
(their (IF1)), which is easier to check than the equivalent definition
from \cite{Hr}.  The next two definitions are taken from \cite{HK}.
\begin{definition}\label{d:hcobs}
  Let $X$ be a proper metric space, and let $\mc{C}$ be the set of
  closed subsets of $X$, and let $d_H$ be the Hausdorff distance on
  $\mc{C}$.  Let $C\in \mc{C}$, let $x\in X$, and let $r$ 
  and $\epsilon$ be positive real numbers. 
  Define
\[ U(C,x,r,\epsilon) = \{D\in \mc{C} \mid d_{H}(C\cap
B(x,r),D\cap B(x,r))\leq
\epsilon\}. \]
  The \emph{topology of Hausdorff convergence
    on bounded sets} is the smallest topology on $\mc{C}$ so that all
  such sets are open.
\end{definition}
\begin{definition}
  Let $X$ be a CAT$(0)$ space, and let $\Gamma$ act properly
  discontinuously, cocompactly, and isometrically on $X$.  The space
  $X$  has \emph{isolated flats} if it contains a
  family of flats $\mathcal{F}$ satisfying:
  \begin{itemize}
    \item \emph{Equivariance:} The set of flats $\mathcal{F}$ is
      invariant under the action of $\Gamma$.  
    \item \emph{Maximality:} There is a constant $B$ such that every flat in $X$ 
      is contained in a $B$--neighborhood of some flat $F'\in \mathcal{F}$. 
    \item \emph{Isolation:} The set $\mc{F}$ is a discrete subset of
      $\mathrm{Flat}(X)$, the set of all flats in $X$ with the
      topology of Hausdorff convergence on bounded sets.
  \end{itemize}
\end{definition}

We prove that our $X$ has isolated flats.
\begin{proposition}\label{p:iflats}
  Let $M(T_1,\ldots,T_m)$ be a $2\pi$--filling of the hyperbolic
  manifold $M$, endowed with the nonpositively curved metric from
  Theorem \ref{t:filling}.  If $X$ is the universal cover of
  $M(T_1,\ldots,T_m)$, then
  $X$ has isolated flats.
\end{proposition}
\begin{proof}
We let $\mc{F}$ be the collection of all flats in $X$.

Maximality and equivariance are clear.  We show the isolation
property.  Let $F_1$ and $F_2$ be elements of $\mc{F}$.  Any path from
$F_1$ to $F_2$ must project to a path from some filling core $V_1$ to
some (possibly the same) filling core $V_2$, which passes through
$\overline{M}'\subset M(T_1,\ldots,T_m)$, where $\overline{M}'$ is as
in the statement of Proposition \ref{p:hyper}.  Such a path must have
length at least $2+\lambda$.

Let $F\in \mc{F}$, and let $x\in F$.  Choose any $r>0$, and let $U =
U(F,x,r,\lambda)$ be defined as in Definition \ref{d:hcobs}.  Clearly $F\in
U$, so $U$ is an open neighborhood of $F$.  On the other hand, since
any path from $F$ to another flat has 
length
at least $2+\lambda$, no other flat is in $U$, so $F$ is an isolated
point of $\mc{F}$.  The flat $F$ was arbitrary, so $\mc{F}$ is discrete.
\end{proof}

\section{Visual boundaries of fillings}\label{s:visbound}
We suppose that $Y=M(T_1,\ldots,T_m)$ is obtained by filling a
hyperbolic $(n+1)$--manifold with $m$ toral cusps, and that the tori $T_i$
all satisfy the hypotheses of Theorem \ref{t:filling}.  It follows
(Corollary \ref{c:cwif}) that the universal cover $X$ of
$M(T_1,\ldots,T_m)$ is a CAT$(0)$ space.
In particular, since $X$ is CAT$(0)$,
the visual boundary $\partial X$ is a boundary for $G=\pi_1(Y)$
in the sense of Bestvina \cite{B}.  The \emph{shape} of
such a boundary is an invariant of $G$ 
\cite[Proposition 1.6]{B}.  
(Since $X$ is either CAT$(-1)$ or CAT$(0)$ with isolated flats,
the \emph{homeomorphism type} of the boundary is an invariant of 
$G$, by a result of Hruska \cite{Hr}.)
In particular, the reduced integral
\v{C}ech cohomology groups of $\partial X$
are invariants of $G$, since they are shape invariants.  
By another result 
of Bestvina, these cohomology groups are the same
as the (reduced) cohomology groups of $G$ with $\Z G$ coefficients
(\cite[Proposition 1.5]{B},
cf. \cite{BM} for the case when $G$ is word hyperbolic).
The particular results we need from \cite{B} can
be summarized:
\begin{theorem}\label{t:bestvina}
\cite{B}
If $G$ is the fundamental group of a compact non-positively curved
space $Y$ with universal cover $X$, and $Z$ is the visual boundary of
$X$, then
\[ \check{H}^q(Z;\Z)\cong H^{q+1}(G; \Z G).\]
Moreover, if $G$ is a Poincar\'e duality group of dimension $n$, then
$Z$ is a \v{C}ech cohomology sphere of dimension $n-1$.
\end{theorem}
In our setting, we do not find the homeomorphism type of $Z$, but we
give enough information to determine the shape of $Z$.

\subsection{The boundary as an inverse limit}\label{s:inverselimit}
If $X$ is any CAT$(0)$ space, and $x\in X$, then the visual boundary
of $X$ is homeomorphic to an inverse limit of metric spheres around
$x$ \cite[II.8.5]{BrHa}.  In order to describe the visual boundary, it therefore suffices
to describe 
\begin{enumerate}
\item the spheres of finite radius about a fixed point, and
\item the projection maps toward the fixed point between these
  spheres.
\end{enumerate}
For the remainder of this subsection,
we suppose that $M$ is a hyperbolic $(n+1)$--manifold, and that 
$M(T_1,\ldots,T_m)$ is a filling satisfying the
hypotheses of Theorem \ref{t:filling}, so that $F_1,\ldots,F_m$ are
the filling cores.  Let $X$ be the universal cover
of $M(T_1,\ldots,T_m)$.
\begin{definition}
If $A$ is a component of the preimage in $X$ of a filling core in
$M(T_1,\ldots,T_m)$, we say that $A$ is a \emph{singular point},
\emph{singular geodesic}, or \emph{singular flat}, depending on
whether $A$ is $0$, $1$, or $k$--dimensional for $k\geq 2$.  The union
of all singular points, geodesics, and flats in $X$ is called the
\emph{singular set}, or $\Xi$.
\end{definition}
\begin{definition}\label{d:Astandard}
If $A$ is a component of the singular set, then by the construction of
Section \ref{s:basic}, $A$ has a regular neighborhood in $X$ isometric
to
\[ [0,\delta]\times_{\cosh(t)}A\times_{\sinh(t)}T^k \]
for some flat torus $T^k$ with injectivity radius strictly 
larger than $\pi$.  We call this a \emph{standard neighborhood of} $A$
of radius $\delta$.
\end{definition}

\begin{remark}
In this section the emphasis is on filling cores which are positive
dimensional, but all statements we give remain true, if the following
conventions are kept in mind:  
\begin{enumerate}
\item The $0$--disk is a point.
\item The $-1$--sphere and the $-1$--disk are both the
empty set $\emptyset$.
\item The spherical join of the empty set with any
metric space $(M,d)$ is taken to be $(M,d_\pi)$, where
$d_\pi(x,y)=\min\{\pi,d(x,y)\}$. 
\end{enumerate}
If all cores are assumed to be zero-dimensional, the results here
merely recapitulate those of \cite{MS}.
\end{remark}

The space $X$ is a manifold away from the singular set $\Xi$.  
Even near the singular set, $X$ is quite well behaved.
\begin{lemma}\label{l:nicelocal}
There exists a constant $\delta$ so that $\linj(p) \geq \delta$ for
all $p\in \Xi$.  
For $p\in X\smallsetminus \Xi$ we have $\linj(p) \geq d(p,\Xi)>0$.
\end{lemma}
\begin{proof}
If $p\in A$
for some singular component $A$ of $\Xi$, and $A$ has a standard
$\delta$--neighborhood as in Definition \ref{d:Astandard}, then
$\linj(p)\geq \delta$, by Corollary \ref{c:loginj}.  Since there are
finitely many orbits of components of $A$ in $X$, some $\delta$ works
for all such $p$.

Let $p\in X\smallsetminus\Xi$,
let $\epsilon>0$, and let $N$ be an open $\epsilon$--neighborhood of
$\Xi$ in $X$.  The space $X\smallsetminus N$ is a complete
nonpositively curved manifold with boundary.  If
$r<d(p,\Xi)-\epsilon$, then the $r$--ball around $p$ lies in the
interior of this manifold, and so the logarithm $\log_p$ restricted to
this ball is injective.  
\end{proof}
\begin{remark}
  We will see in Lemma \ref{l:reverseshadow} that if $p\notin \Xi$,
  then $\linj(p)\leq d(p,\Xi)$, so actually $\linj(p) = d(p,\Xi)$.
\end{remark}
\begin{lemma}\label{l:gepX}
  $X$ has the geodesic extension property.
\end{lemma}
\begin{proof}
  Away from the singular set, $X$ is a Riemannian
  manifold of nonpositive curvature, so local geodesics can be
  extended in $X\setminus \Xi$.  A local geodesic terminating on $\Xi$
  can be extended by Lemma \ref{l:gep}.
\end{proof}

For the remainder of the section, we fix a point
$x_0\in X\setminus \Xi$, and a constant $\delta>0$ as in Lemma
\ref{l:nicelocal}. 
\begin{definition}
A \emph{sphere centered at $x_0$} is a metric sphere
\[ S_r = \{x\in X\mid d(x,x_0)=r\}\]
for some $r>0$.  We write $B_r$ for the closed ball of radius $r$
centered at $x_0$, and write $\mathring{B}_r$ for the interior of $B_r$.
\end{definition}
For small enough $r$, $S_r$ is homeomorphic to an ordinary
sphere $S^n$ and $B_r$ to an ordinary $(n+1)$--dimensional ball.
Describing the topology for larger $r$ is the main goal in this subsection.

Recall that an \emph{ENR}, or \emph{Euclidean Neighborhood Retract}, is
any space homeomorphic to a retract of an open subset of some
euclidean space.  Equivalently, an ENR is any locally compact,
metrizable, separable, finite dimensional ANR.
\begin{lemma}\label{l:enr}
Suppose that $d(x_0,A)\neq r$ for any component $A$ of the singular
set $\Xi$.  Then the metric sphere $S_r$ is an ENR, and $S_r\smallsetminus
\Xi$ is an $n$--manifold which is open and dense in $S_r$.
\end{lemma}
\begin{proof}
We first note that $S_r$ is finite dimensional, as it is embedded in
the $(n+1)$--dimensional space $X$.  It is obviously locally compact,
metrizable, and separable, so we need only establish that $S_r$ is an
ANR to prove the first part of the lemma.
A theorem of Dugundji \cite{Du} states that a finite dimensional
metric space is an ANR if and only if it is locally contractible.  We
therefore examine the local behavior of $S_r$.
Let $p\in S_r$.  
\begin{claim}
  There is some $q\in\mathring{B}_r$ so that $p$ is contained in the
  interior of a
  log-injective ball centered at $q$.
\end{claim}
\begin{proof}
  There are two cases, depending on whether $p$ is in the singular set $\Xi$
  or not.  Suppose first that $p\notin \Xi$, and let $d =
  d(p,\Xi)$.  Let $q\neq p$ be a point on the geodesic from $p$ to $x_0$,
  so that $d(p,q)<d/3$.  The point $q$ is no closer than
  $\frac{2}{3}d$ from $\Xi$, so the log-injectivity radius at $q$
  is at least $\frac{2}{3} d< d(p,q)$, by Lemma \ref{l:nicelocal}.  It
  follows that $p$ is in a log-injective ball about $q$.

  In the second case, $p$ is contained in some component $A$ of the
  singular set.  By the hypothesis of the lemma $d(x_0,A)< r$; let
  $z\in A$ satisfy $d(x_0,z)<r$.  Choose $q\neq p$ on the geodesic from $p$ to
  $z$ so that $d(p,q)<\delta$, where $\delta$ is as in Lemma
  \ref{l:nicelocal}.  
  Since $A$ is convex, $q\in A$.  Applying Lemma \ref{l:nicelocal}, 
  the log-injectivity radius at $q$ is at least $\delta$, and so $p$
  is contained in a log-injective ball centered at $q$.
\end{proof}
Let $U$ be the log-injective ball around $q$ given by the claim, and
let $B$ be a slightly smaller closed
ball around $q$ so that $p$ is still in
the interior of $B$.  
The intersection $C=B\cap B_r$ is a strictly
convex set containing $q$ in its interior, so that $C$ is contained in
a log-injective neighborhood of $U$.  Moreover, $X$ has the geodesic
extension property by Lemma \ref{l:gepX}.
We can therefore apply Lemma
\ref{l:dconvex} to obtain a homeomorphism 
\[ h\co \mathrm{Front}(C)\to \Sigma_q\]
from the frontier of $C$ to
the space of directions $\Sigma_q$ at $q$.  This space of directions
is always either a round $n$--sphere or (using Lemma \ref{l:dirjoin}) 
an $n$--pseudomanifold isometric to the join of
a sphere and a torus.  In particular, it is locally contractible.  

Let $s$ be the radius of the ball $B$, and choose a positive 
$\epsilon < s-d(p,q)$.  The open $\epsilon$--ball around $p$ intersects the
frontier of $C$ in a set $V\subset S_r$, which is an open neighborhood
of $p$ in $S_r$.  Using the homeomorphism $h$ to the locally
contractible space $\Sigma_q$, we may find an open contractible
$V'\subset V$ containing $p$.

To see that $S_r\setminus \Xi$ is a manifold, note that, in the
argument just made, if $p\notin \Xi$, then $q\notin\Xi$, and $h$
takes a neighborhood of $p$ in $S_r$ to an open set in the sphere
$S^n$.  Even for $p\notin \Xi$, the local description makes it clear
that $p$ is in the closure of the set of manifold points of $S_r$.
\end{proof}

\begin{definition}
For any $r$, there is a map $p_r$ from the exterior of
$S_r$ to $S_r$ defined by geodesic projection toward $x_0$:  
If $x$ is any point lying further than $r$ from
$x_0$, then the geodesic from $x$ to $x_0$ intersects $S_r$ in the
unique point $p_r(x)$.    
\end{definition}

\begin{definition}\label{d:good}
A \emph{shell} is a set of the form
\[ S_a^b = \{x\in X\mid a\leq d(x,x_0)\leq b\} = B_b\setminus \mathring{B}_a.\]
A shell $S_a^b$ is \emph{good} if it satisfies all of the following:
\begin{enumerate}
\item\label{reg} If $\delta$ is the constant fixed after Lemma
  \ref{l:nicelocal}, then $b-a<\delta$.
\item\label{goodint} Whenever $A$ is an $l$--dimensional component
  of $\Xi$ and $c\in \{a,b\}$, then $A\cap S_c$  is either
  empty or a (topological) sphere of dimension $l-1$.
\item\label{embed} The map $p_a$ restricted to $\Xi\cap S_a^b$ is
  an embedding. 
\end{enumerate}
\end{definition}
The following observation will be useful later.
\begin{lemma}\label{l:subgood}
Suppose that $a\leq s < t\leq b$.  If $S_a^b$ is good, then
either $S_s^t$ is good, or $d(x_0,A)\in \{s,t\}$ for some
component $A$ of the singular set $\Xi$.
\end{lemma}
Components of the singular set have very nice intersections with
good shells:
\begin{lemma}\label{l:intersect}
Suppose $S_a^b$ is a good shell.
Let $A$ be a component of the singular set of dimension $l$ which
intersects $S_a^b$ nontrivially.
If $A$
intersects $S_a^b$ but not $S_a$, then $A\cap S_a^b$ is a disk of
dimension $l$.  Otherwise $A\cap S_a$ is a sphere of dimension $l-1$ and
there is a homeomorphism $h_{a,b,A}$
from $A\cap S_a^b$ to $(A\cap S_a) \times I$.
\end{lemma}
\begin{proof}
Distance to the basepoint $d(-,x_0)$ is a proper strictly convex
function on
$A\cong \E^l$.  If $A$ does not intersect $S_a$, then $A\cap S_a^b$ is a
sublevel set of this strictly convex function, so $A\cap S_a^b$ is
strictly convex.  Since $d(A,x_0)<b$, the set $A\cap S_a^b$ has
nonempty interior (as a subset of $A$), so it is an $l$--disk.

Now suppose that $A$ intersects $S_a$ as well as $S_b$.  Let $q$ be the
closest point on $A$ to $x_0$.  We define the first coordinate of 
\[ h_{a,b,A}\co A\cap S_a^b \to ( A\cap S_a ) \times I \]
by geodesic projection toward $q$; the second coordinate is the
distance to $x_0$, affinely rescaled to lie in $[0,1]$.  Explicitly,
\[ h_{a,b,A}(p) = \left( z , \frac{d(p,x_0)-a}{b-a} \right), \]
where $z$ is the unique point in $S_a$ on a geodesic from $p$ to $q$.
The convexity of $A$ and of the distance function make this map a
homeomorphism.
\end{proof}
The next lemma says that there are plenty of good shells.
\begin{lemma}\label{l:good}
For every $r>0$ there is some $\epsilon>0$ so that the shell
$S_{r-\epsilon}^{r+\epsilon}$ is good.
\end{lemma}
\begin{proof}
Let $\delta$ be the smallest radius of a standard neighborhood of a
component of the singular set of $X$.  The distance between two
distinct components is thus greater than $2\delta$.  
It follows that if $\epsilon< \delta/2$, then
\[p:=p_{r-\epsilon}|_{\Xi\cap S_{r-\epsilon}^{r+\epsilon}}\]
cannot send points from different
singular components to the same point in $S_{r-\epsilon}$.

Let $\mathcal{A}$ be the collection of components of the singular set
$\Xi$, and note that the set 
\[S:=\{s\mid \exists A\in \mathcal{A},\ d(x_0,A)=s\}\] 
is a discrete subset of $[0,\infty)$.  We may therefore choose
$\epsilon < \min\{\delta/2,r\}$ so that neither $r-\epsilon$ nor 
$r+\epsilon$ is in $S$.  Condition \eqref{reg} is obviously satisfied for any
such $\epsilon$.  We will show that \eqref{goodint} and
\eqref{embed} are as well.

Suppose $A\in\mathcal{A}$ intersects $S_c$ nontrivially for $c\in
\{r+\epsilon,r-\epsilon\}$, and let $B = B_c\cap A$ be the
intersection of $A$ with the ball of radius $c$ around $x_0$.
Since $c > d(A,x_0)$, the set $B$ has non-empty interior (as a
subset of $A$), and since both $A$ and the metric are convex, $B$ is
a strictly convex subset of $A$.  Lemma \ref{l:dconvex} implies that
$\partial B = A\cap S_c$ is an $(l-1)$--sphere, so condition
\eqref{goodint} of Definition \ref{d:good} is satisfied.

Finally, suppose that
that $S_{r-\epsilon}^{r+\epsilon}$ does not satisfy condition
\eqref{embed} of Definition \ref{d:good}.  There
is then some $A \in \mathcal{A}$ and two points $x$, $y\in A\cap
S_{r-\epsilon}^{r+\epsilon}$ so that $p(x)=p(y)$.  Since $p(x)=p(y)$,
the geodesics $[x_0,x]$ and $[x_0,y]$ must coincide on some (maximal) initial
segment $[x_0,z]$ of length at least $r-\epsilon$.

There are two cases:

\noindent
{\bf Case 1:}
Suppose that $z\in \{x,y\}$.  Without loss of generality, suppose that $z=x$,
and rechoose $x$ so that $x$ is the closest point on
$[x_0,x]\cap A=\{x\}$.  Since $A$ is convex,
the geodesic $[x,y]$ lies entirely in $A$, so $[x,y]$ represents a
point in $S^{n-k-1}\subset S^{n-k-1}\join T^k = \Sigma_x$.
Since $[x_0,y]$ is geodesic, the geodesic
sub-segments $[x,y]$ and $[x,x_0]$ must
have Alexandroff angle $\pi$ at $x$.  This implies that $[x,x_0]$ also
lies in $S^{n-k-1}\subset S^{n-k-1}\join T^k = \Sigma_x$, but since $A$
has a standard neighborhood, this implies that $[x,x_0]$, and hence
$x_0$, lies in $A$.  But this contradicts the fact that $x_0$
lies outside the singular set.

\noindent
{\bf Case 2:}
Suppose $z\notin \{x,y\}$.  The directions in $\Sigma_z$ corresponding
to $[z,x]$ and $[z,y]$ are both at distance $\pi$ from the direction
corresponding to $[z,x_0]$.  Since the logarithm map is injective near
$z$, these directions are all distinct in $\Sigma_z$.  In particular,
$\Sigma_z$ is not a round sphere of diameter $\pi$ and so $z$ must lie in some
component of the singular set $A'$.
Since $p(z)=p(x)=p(y)$, we must have $A'=A$.
Replacing either $x$ or $y$ by $z$, we may derive a contradiction as
in Case 1.

This establishes condition \eqref{embed}.
\end{proof}

To understand the difference between the two boundary components of a
good shell, we have to understand how geodesics can be continued
through a piece of the singular set.
In a simply connected Riemannian manifold of non-positive curvature, geodesics 
(in the sense of length-minimizing paths) can always be
continued uniquely.  At a point in a singular space
whose space of directions is not a round sphere of diameter $\pi$ 
(or a point which does not have a log-injective
neighborhood) this uniqueness generally fails.
The next lemma gives a description of the collection of geodesics of a
fixed length starting at a point and passing through some particular
component of the singular set.
\begin{lemma}\label{l:reverseshadow}
Let $A$ be a component of the singular set isometric to $\E^l$, and
let $k=n-l$.  Choose $R$ so that
\[ d(x_0,A) < R < d(x_0,A)+\delta.\] 
If $Z$ is the set of endpoints of geodesics of length $R$ starting at
$x_0$ and meeting $A$, then $Z$ is homeomorphic to $ S^{l-1} \join P^k$,
where $P^k$ is a flat $k$--dimensional torus minus the interior of a
closed embedded $k$--dimensional ball.
\end{lemma}
\begin{proof}
  By the construction of
Section \ref{s:basic}, there is a standard neighborhood of $A$
which we may identify with the warped product:
\[[0,\delta]\times_{\cosh(r)}A\times_{\sinh(r)}T \]
for some flat $k$--torus $T$ whose injectivity radius is greater than $\pi$.
Let $a\in A$ be the closest point to $x_0$, and let $a'\in A$ be arbitrary.  Via
Proposition \ref{p:dirwarp}, we may identify the space of directions
$\Sigma_{a'}(X)$ with 
$\Sigma_a(A)\join T$.  
Let 
\[\pi_T  \co \Sigma_{a'}(X)\smallsetminus \Sigma_{a'}(A) \to T \]
be the canonical projection onto the $T$ factor.  
Finally, let $\psi\co Z\to A$ send a point $z\in Z$ to the unique point
in $A$ on the geodesic from $z$ to $x_0$.

We can then define a map
$\phi\co Z\to \Sigma_a(A)\join T = [0,\frac{\pi}{2}]\times
\Sigma_a(A)\times T / \sim$
by 
\[\phi(z) = \left(\frac{\pi}{2}\frac{d(\psi(z),z)}{R-d(x_0,a)},
      [a,\psi(z)], \pi_T([\psi(z),z])\right),\]
abusing notation in the second and third factor by writing a geodesic
segment with a given direction instead of the direction itself.
Because geodesics are continuous as their endpoints are varied, $\phi$
is a continuous map onto some closed subset of $\Sigma_a(A)\join T$.
Note
here that the first two factors are essentially the polar coordinates
centered at $a$ of $\psi(z)$, reparametrized in the radial direction.
To check injectivity of $\phi$, we therefore only need to check
injectivity of $\phi$ restricted to $\psi^{-1}(a')$ for $a'\in A\cap
B_R$.  If $d(a',x_0)=R$, there is 
nothing to check, so it suffices to verify the following.
\begin{claim}
For $a' \in A$ with $d(a',x_0)<R$, $\phi$ restricted to
$\psi^{-1}(a')$ is an embedding, and $\psi^{-1}(a')$ is homeomorphic
to $T$ minus an open ball of radius $\pi$.
\end{claim}
\begin{proof}
For any $z\in \psi^{-1}(a')$, the geodesic from $x_0$ to $z$ is
composed of two subsegments $[a',x_0]$ and $[a',z]$.  Regarded as
directions in $\Sigma_{a'}(X)$, these subsegments must subtend an
angle of $\pi$.  Conversely, any direction which is Alexandroff angle
$\pi$ from $[a',x_0]$ is the direction of the geodesic continuation of
$[x_0,a']$ on to some unique $z\in Z$.  (Uniqueness follows from
log-injectivity at $a'$ of the $\delta$--neighborhood of $A$, in which
$Z$ is contained.)  To sum up, $\log_{a'}$ is an embedding restricted
to $\psi^{-1}(a')$, with image $D$, the set of directions whose angle with
$[a',x_0]$ is exactly $\pi$.

Identifying $\Sigma_{a'}(X)$ with $\Sigma_a(A)\join T$, suppose that
$[a',x_0] = (\phi,\alpha,\theta)$ in the standard join coordinates.
Because $x_0\notin A$, $\phi\neq 0$.  Standard trigonometric
identities can be used to show that those $(\phi',\alpha',\theta')$
which make
an angle of $\pi$ with $[a',x_0]$ are exactly those with
 $\phi'= \phi$, $\alpha'=-\alpha$, and $d(\theta',\theta)\geq \pi$.
It follows that $\pi_T(\log_{a'}(\psi^{-1}(a')))$ is $T$ minus an open
ball of radius $\pi$ around $\theta$.
\end{proof}
It follows that the image of $\phi$ is homeomorphic to the join of
$\Sigma_a(A)$ with $T$ minus a ball of radius $\pi$.  Since $\phi$ was
a homeomorphism onto its image, and $\Sigma_a(A)$ is an
$(l-1)$--sphere, the lemma is proved.
\end{proof}
\begin{remark}
  The space $S^{l-1}\join P^k$ described in Lemma
  \ref{l:reverseshadow} is a pseudomanifold with boundary homeomorphic
  to $S^{l-1}\join S^{k-1}\cong S^{n-1}$.
\end{remark}

\begin{definition}\label{d:connectsum}
Let $Y_1$ and $Y_2$ be spaces containing $U_1\subseteq Y_1$ and
$U_2\subseteq Y_2$ which are open, connected, dense, and homeomorphic
to oriented $n$--manifolds.  Then the \emph{connect sum} $Y_1\mathrel{\#} Y_2$ is the space
obtained by choosing small closed $n$--balls $D_1\subset U_1$ and
$D_2\subset U_2$, choosing an orientation reversing homeomorphism $\phi\co \partial
D_1\to\partial D_2$, and gluing the exteriors together thus:
\[ Y_1\mathrel{\#} Y_2 := (Y_1\smallsetminus\mathring{D}_1)\cup_\phi
(Y_1\smallsetminus\mathring{D}_2) .\]
For $\{i,j\}=\{1,2\}$ there is a (unique up to homotopy) map $q_{Y_i}\co
Y_1\mathrel{\#}Y_2\to Y_i$ which takes
$Y_i\smallsetminus\mathring{D}_i$ to itself by the identity map, and
sends $Y_j\smallsetminus\mathring{D}_j$ onto $D_i$ by a degree one
map.  We refer to $q_{Y_i}$ as \emph{the map which pinches $Y_j$ to a point.}
\end{definition}

We remark that the operation of connect sum described in Definition
\ref{d:connectsum} is associative, so that the connect sum of three or
more spaces is well-defined.  Moreover, if
$W=(A\mathrel{\#}B)\mathrel{\#}C \cong A\mathrel{\#}(B\mathrel{\#}C)$, then
the map which pinches $B\mathrel{\#}C$ to a point is homotopic to the
composition of the map which pinches $B$ to a point with the map which
pinches $C$ to a point.

\begin{proposition}\label{p:pinch}
Suppose $S_a^b$ is a good shell, and let
$A_1,\ldots,A_p$ be the
components of the singular set which intersect $S_a^b$ but not $S_a$.
If the dimension of $A_i$ is $l_i=n-k_i$ for each $i$, then there is a
homotopy equivalence $\phi$ from
$S_b$ to the connect sum $S_a\mathrel{\#} J$, where
\begin{equation}
J = \mathop{\#}_{i=1}^p S^{l_i-1}\join T^{k_i}.
\end{equation}
Moreover,
the projection map $p_a^b = p_a|_{S_b}$ fits into a homotopy
commutative triangle
\cd{S_b\ar[dr]_{p_a^b}\ar[rr]^\phi & & S_a\mathrel{\#} J\ar[dl]^{q_{S_a}}\\
& S_a & }
where $q_{S_a}$ is the map which pinches $J$ to a point.
\end{proposition}
\begin{proof}
One complicating point here is that near the components of the
singular set, the projection map $p_a^b$ is not a local
homeomorphism.  This is true even if $S_a$ and $S_b$ intersect the
same components of the singular set.  
To deal with this issue, we replace $S_a$ and
$S_b$ with homotopy equivalent quotients $Q_a$ and $Q_b$ so that
$Q_b\cong Q_a\mathrel{\#}J$, chosen so that $p_a^b$ induces a map
$\bar{p}_Q\co Q_b\to Q_a$, and so that $p_Q$
is the map which pinches $J$ to a point.

The maps $h_{a,b,A}$ from Lemma \ref{l:intersect}
can be patched together to give an embedding
$h\co (S_a\cap \Xi) \times I\to S_a^b$ whose image is the union of
the components of $\Xi\cap S_a^b$ which intersect both boundary
components of the shell.
Since $S_a^b$ is a good shell, the map $p_a$ restricts to an
embedding of $\Xi \cap S_a^b$, and so $p_a\circ h$ is also an embedding.

We next give decompositions $\mc{G}_a$ and $\mc{G}_b$ of $S_a$ and
$S_b$ into closed sets.  A set in $\mc{G}_a$ is either a point outside
the image of $p_a\circ h$, or it is an arc of the form $p_a\circ h(e\times
I )$ for some $e\in S_a\cap\Xi$.  

A set in $\mc{G}_b$ is either a
point $z$ so that $p_a(z)$ lies outside the image of $p_a\circ h$, or it
is the preimage of one of the arcs of the form $p_a\circ h(e\times I)$.
Arguing as in Lemma \ref{l:reverseshadow}, each element of
$\mc{G}_b$ is either a point or homeomorphic to the cone on a closed
subset of a torus.

We leave it to the reader to check that the decompositions $\mc{G}_a$
and $\mc{G}_b$ are 
upper semicontinuous. It follows that $Q_a$ and $Q_b$ are Hausdorff, and
therefore compact metric spaces.

The projection map from $S_a$ sends elements of $\mc{G}_b$ to elements
of $\mc{G}_a$, and thus induces a continuous map $p_Q$ from $Q_b$
to $Q_a$.  We have the following commuting square:
\cdlabel{square}{S_b \ar[r]^{q_b}\ar[d]^{p_a^b} & Q_b\ar[d]^{p_Q} \\
           S_a \ar[r]^{q_a} & Q_a}
Note that the map $p_Q$ is a homeomorphism away from the image in
$Q_a$ of 
\[(p_a^b)^{-1}\left(p_a(\cup_iA_i\cap S_a^b)\right).\]
The Proposition now reduces to two claims:
\begin{claim}\label{vert}
The space $Q_b$ is homeomorphic to the connect sum $Q_a\consum J$,
and $p_Q$ is homotopic to the map which pinches $J$ to a point. 
\end{claim}
\begin{claim}\label{horiz}
The horizontal maps in \eqref{square} are homotopy equivalences.
\end{claim}

\begin{proof}[Proof of Claim \ref{vert}]
As noted above, $p_Q$ is a homeomorphism away from
\[q_a\left((p_a^b)^{-1}(\cup_{i=1}^p A_i\cap S_a^b)\right),\]
where the $A_i$ are those components
of the singular set which intersect $S_a^b$ but not $S_a$.
A small regular neighborhood in $S_a$ of each $l_i$--dimensional disk 
$D_i:=p_a (A_i\cap S_a^b)$ is a ball $B_i$
(again using Lemma \ref{l:intersect}), so this implies that $Q_b$ is
homeomorphic to the connect sum of $Q_a$ with some collection of
spaces $\{K_i\mid 1\leq i\leq p\}$ where each $K_i$ is equal to
$(p_a^b)^{-1}(B_i)$, capped off with a ball.

We must describe the topology of the spaces $(p_a^b)^{-1}(B_i)$.
The topology of $(p_a^b)^{-1}(D_i)$ is given in Lemma
\ref{l:reverseshadow}:  Each such space is homeomorphic to a join
$S^{{l_i}-1} \join P^{k_i}$ with $k_i+l_i=n$ and $P$ homeomorphic to a
$k_i$--torus with a small open ball removed. 
It follows that $K_i$ is homeomorphic to the
join of a $k_i$--torus with an $(l_i-1)$--sphere.

Putting the pieces together, we see that $Q_b$ is homeomorphic to a
connect sum of $Q_a$ with the space $J$ described in the statement, and
projection from $Q_b$ to $Q_a$ is homotopic to the map which pinches
$J$ to a point.  
\end{proof}

\begin{proof}[Proof of Claim \ref{horiz}]
We first claim that the quotient spaces $Q_a$ and $Q_b$ are ENRs.
Indeed, Lemma \ref{l:enr} implies that the metric spheres $S_a$ and
$S_b$ are ENRs, so
each of $Q_a$ and $Q_b$ is a quotient of an ENR by a
decomposition into closed contractible (in particular cell-like)
sets.  A theorem of Lacher \cite[Corollary 3.3]{La} says that if such
a quotient is metrizable and finite dimensional, then it is an ENR.
We have already shown that $Q_a$ and $Q_b$ are metrizable.  Moreover,
there is an obvious decomposition of each into an open $n$--manifold
and a union of closed annuli and disks of strictly lower
dimension. (Only annuli are needed for $Q_a$.)
Applying the Addition Theorem from dimension theory (see,
e.g. \cite[3.1.17]{En}), we obtain the finite dimensionality of $Q_a$
and $Q_b$.

It now follows from another theorem of Lacher \cite[Theorem 1.2]{La}
that the quotient maps are homotopy equivalences.
\end{proof}

The Proposition now follows from the two claims.
\end{proof}

\begin{theorem}\label{t:cohomology}
Let $n\geq 2$.
Let $M$ be a hyperbolic $(n+1)$--manifold with $m$ toral cusps
$E_1,\ldots,E_m$, let $\overline{M}$ be a manifold with boundary obtained from $M$
by removing open horospherical neighborhoods of the cusps,
and let 
$\{T_1,\ldots,T_m\}$ be totally geodesic tori in $\partial\overline{M}$ satisfying
the hypotheses of Theorem \ref{t:filling}.  Let $s$ be the maximum of
the dimensions of $\{T_1,\ldots,T_m\}$.
Let $X=M(T_1,\ldots,T_m)$ be obtained from $M$ by filling along
$\{T_1,\ldots,T_m\}$, and let $G = \pi_1 (X)$.

The cohomology of $G$ with $\Z G$ coefficients satisfies:
\[H^q( G;\Z G)=
\begin{cases}
  0 & \quad q \leq n-s+1\mbox{ or }q>n+1\\
  \Z^\infty & \quad n-s+2 \leq  q \leq n\\
  \Z & \quad q=n+1.
\end{cases}
\]
\end{theorem}
\begin{proof}
Let $Y$ be the universal cover of $X$.  By Theorem \ref{t:filling},
the space $Y$ is CAT$(0)$.  As pointed out in
\cite{B}, the visual boundary $\partial Y$ is a
boundary for $G$, and so we have 
$H^q(G;\Z G)=\check{H}^{q-1}(\partial Y;\Z)$ for all $q$.

We thus compute the \v{C}ech cohomology of $\partial Y$.  

As at the beginning of Section \ref{s:inverselimit}, choose
some $x_0$ in $Y$, not in the singular set, and write $S_r$ for
the metric sphere of radius $r$ centered as $x$.
As already noted,
the visual boundary of $Y$ is homeomorphic to the
inverse limit of the system of spheres centered at $x$, i.e.,
\[\partial Y \cong \invlim\left\{S_t\stackrel{p_s^t}{\longrightarrow} S_s\mid s, t\in \R\quad\mbox{and}\quad
t\geq s\right\}\]
where $p_t^s$ is projection along geodesics toward $x_0$.  In fact,
any sub-collection of spheres $\{S_{r_i}\mid i\in \N\}$ with
$\lim_{i\to\infty}r_i=\infty$ determines the inverse limit.
By Lemmas \ref{l:subgood} and \ref{l:good}, we may therefore
choose the sequence $\{r_i\}$ so that each shell $S_{r_i}^{r_{i+1}}$
is good.  Let $p_i = p_{r_{i}}^{r_{i+1}}$.  We have
\[\partial Y
\cong\invlim\left\{S_{r_{i+1}}\stackrel{p_i}{\longrightarrow}S_{r_i} \mid 
i\in \N\right\},\]
and so
\[\check{H}^q(\partial Y; \Z) =
\dirlim
\left\{H^q(S_{r_i};\Z)\xrightarrow{{p_i^*}}H^q(S_{r_{i+1}};\Z)\mid
i\in \N \right\}.\]
For each $i$, we apply Proposition \ref{p:pinch} to compute $p_i^*$.
Let $J_i$ be the space $J$ from the proposition, applied with $b =
r_{i+1}$ and $a=r_i$, and let $s_i$ be the largest dimension of any
torus $T_j$ occurring in $J_i= \mathop{\#}_{j=1}^{p_i} S_j\join T_j$.
Elementary algebraic topology shows that $p_i^*$ is injective, and
that $\coker(p_i^*)$ is free
abelian of nonzero rank in dimension $j$ for 
\[n-s_{i}+1\leq j\leq n-1,\]
and zero in all other dimensions.  Note that
$s_i = s$ for infinitely many $i\in \N$.

It follows that the (reduced) \v{C}ech
cohomology of $\partial G$ is 
\[\check{H}^q(\partial G;\Z)=\left\{\begin{array}{ll}
 0 & \mbox{if}\quad q \leq n-s\mbox{ or } q>n\\
 \Z^\infty & \mbox{if}\quad n-s+1 \leq  q \leq n-1\\
 \Z & \mbox{if}\quad q=n.
\end{array}\right.
\]

Applying Theorem \ref{t:bestvina} gives the desired result.
\end{proof}

\begin{corollary}\label{c:notpdn}
Let $M$, $\{T_1,\ldots,T_m\}$ satisfy the hypotheses of Theorem
\ref{t:cohomology}. 
The group $G=\pi_1(M(T_1,\ldots,T_m))$ is a Poincar\'e duality group
if and only if every $T_i$ has dimension one.
\end{corollary}
\begin{proof}
If every $T_i$ has dimension one, then $M(T_1,\ldots,T_m)$ is a
non-positively curved (and hence aspherical) closed manifold of
dimension $n+1$, and so $\pi_1(M(T_1,\ldots,T_m))$ must be a
Poincar\'e duality group of dimension $n+1$.  (The special case in which every $T_i$ has dimension one was already
handled by Schroeder in \cite{Sc}.)

If some $T_i$ has dimension different from one, then Theorem
\ref{t:cohomology} implies that $G$ has cohomological dimension $n+1$, but
$H^{q}(G;\Z G)\neq 0$ for some $q\neq n+1$, and so $G$ cannot be an
$(n+1)$--dimensional Poincar\'e duality group, by the second part of
Theorem \ref{t:bestvina}.
\end{proof}
\begin{remark}
  The proof that our examples are not manifold groups is
  essentially the same as that used in \cite{MS} for the case in which
  every cusp is coned off completely.
  The difference in
  our setting is that we get a wider variety of behavior at infinity.
\end{remark}
\subsection{Simple connectivity at infinity}

\begin{proposition}\label{p:simpconn}
  Let $M(T_1,\ldots,T_m)$ be a $2\pi$--filling of an
  $(n+1)$--dimensional hyperbolic manifold $M$.  Unless some $T_i$ is
  $n$--dimensional, the universal cover $X$ of $M(T_1,\ldots,T_m)$ is
  simply connected at infinity.
\end{proposition}
\begin{proof}  
  We retain the notation of Subsection \ref{s:inverselimit}.  
  To show simple connectivity at infinity, it
  suffices to show that for every $r>0$ there is some $R>r$ so that
  any loop in $X\setminus B_R$ is null-homotopic in $X\setminus B_r$.

  Accordingly, let $r>0$, and 
  choose $r_1<r_2<r_3<\cdots$  as in the proof of Theorem
  \ref{t:cohomology}, so that each shell $S_{r_i}^{r_{i+1}}$ is
  good and so that $\lim_{i\to\infty} r_i=\infty$.  Choose some $i$ so
  that $r_i>r$.  We claim that $R = r_i$ suffices.

  Let $\gamma\co S^1\to X\setminus B_R$ be a loop.  It is easy to show
  that $X$ deformation retracts (holding $B_R$ fixed) to $B_R$, so
  $\gamma$ is homotopic in the complement of $B_r$ to a loop $\gamma'$
  in
  $S_R$.  
  Induction together with Proposition
  \ref{p:pinch} shows that $S_R$ is a connect sum of joins of tori and
  spheres.  Moreover,
  unless some $T_i$ is $n$--dimensional,
  each such join is simply connected, so
  $S_R$ is simply connected.  It follows that $\gamma'$ is
  null-homotopic in $S_R$, so $\gamma$ is null-homotopic in the
  complement of $B_r$.
\end{proof}

In \cite{Osa}, it is shown that systolic groups are never simply
connected at infinity. 
\begin{corollary}\label{c:notsys}
Let $G = \pi_1(M(T_1,\ldots,T_m))$, where $M(T_1,\ldots,T_m)$ is a $2\pi$--filling of an
$(n+1)$--dimensional hyperbolic manifold, and suppose that no $T_i$ is
$n$--dimensional. Then $G$ is not systolic.
\end{corollary}

\section{Further questions}\label{s:questions}
Suppose we are given a group $G$ and a collection of subgroups
$\mathcal{P}=\{P_1,\ldots,P_m\}$ with respect to which $G$ is
relatively hyperbolic.  Suppose further that normal subgroups $N_i\lhd
P_i$ (filling kernels)
are given, and $K = \llangle \cup_i N_i\rrangle_G$ is the normal
closure in $G$ of these subgroups.  The group $G(N_1,\ldots,N_m):=G/K$
is said to be a \emph{hyperbolic filling of $G$} if 
\begin{enumerate}
\item the obvious map $\psi_i\co P_i/N_i\to G(N_1,\ldots,N_m)$ is
  injective for each $i$, and
\item $G(N_i,\ldots,N_m)$ is hyperbolic relative to
  $\{\psi_i(P_i/N_i)\mid 1\leq i\leq m\}$.
\end{enumerate}

The main results of \cite{GM,Os} show that as long as
every filling kernel $N_i$ is composed entirely of ``long'' elements
of $P_i$, the group $G(N_1,\ldots,N_m)$ is always a hyperbolic filling
of $G$.

If $G$ acts properly and cocompactly on a CAT$(0)$ space with isolated
flats, then it is hyperbolic relative to the flat stabilizers
\cite{HK}.  These flat stabilizers are virtually abelian,
and so there are many candidates for the filling kernels as above.

\begin{question}
If $G$ acts properly and cocompactly on a CAT$(0)$ space with isolated
flats, does every hyperbolic filling of $G$ act properly and
cocompactly on some CAT$(0)$ space with isolated flats?  What can be said
about the dimension of this space?
\end{question}

\begin{question}
If $G$ acts properly and cocompactly on a CAT$(0)$ space with isolated
flats (or geometrically finitely on a CAT$(-1)$ space),
and $G'$ is a hyperbolic filling of $G$ which is hyperbolic,
does $G$ act properly and cocompactly on some CAT$(-1)$ space?  What
can be said about the dimension of this space?
\end{question}

In the current paper we have given answers to both of these questions in the
special case that $G$ is the fundamental group of a hyperbolic
$(n+1)$--manifold with toral cusps
and the filling kernels are direct summands of the
(free abelian)
peripheral subgroups.  It should be possible to appeal to Haefliger's theory
of nonpositively curved orbispaces and extend our results to the case
in which the filling kernels are arbitrary subgroups of the peripheral
subgroups.  

It seems much more challenging to extend the current techniques to hyperbolic
fillings of fundamental groups of finite volume
complex or quaternionic hyperbolic manifolds.

Focusing on fillings of the fundamental group $G$ of a fixed hyperbolic
manifold $M$ with toral cusps, other kinds of question
arise.  Theorem \ref{t:examples} says that there are infinitely many
isomorphism types of hyperbolic fillings of $M$.  On the other hand,
the discussion in Section \ref{s:visbound} shows that these fillings
have boundary whose shape is determined by the dimensions of the
filling tori.
\begin{question}\label{q:diff}
Let $n\geq 2$.
Let $M$ be a hyperbolic $(n+1)$--manifold with toral cusps, and 
let $A\subseteq \{1,\ldots,n\}$.  Let $\mc{M}_A$ be the collection of
fundamental groups of $2\pi$--fillings $M(T_1,\ldots,T_m)$ so that
\[\left\{\dim(T_i)\mid
i\in\{1,\ldots,m\}\right\}=A.\]
How many quasi-isometry types are there in $\mc{M}_A$?
\end{question}
The case of classical hyperbolic Dehn filling of $3$--manifolds is $A
= \{1\}$ and $n = 2$.  The $2\pi$ theorem together with the positive
solution to the geometrization conjecture by Perelman \cite{MoTi1,MoTi2} implies
that, in this case, every group in $\mc{M}_A$ is quasi-isometric to $\H^3$.

In Mosher and Sageev's setting in \cite{MS}, $A
= \{n\}$ and the set $\mc{M}_A$ has at most one element.  

In all other cases,  Question \ref{q:diff} seems to be open.

\providecommand\url[1]{\texttt{#1}}

\end{document}